\documentclass{article}
\usepackage{amsmath,amssymb}
\usepackage{amscd}
\usepackage[a4paper,nohead,headsep=0.6cm,left=2cm,right=2cm,top=2cm,
bottom=2cm,twoside]{geometry}
\usepackage{amsthm}
\usepackage{euscript}
\usepackage{latexsym}
\usepackage[matrix,arrow,curve,cmtip]{xy}
\usepackage[latin1]{inputenc}
\usepackage[english]{babel}
\usepackage{wrapfig}
\usepackage{mathrsfs}
\usepackage[pdftex]{graphicx}
\usepackage{keyval}

\usepackage{tikz}
\usetikzlibrary{shapes,arrows}

\usepackage{verbatim}

\newtheorem{theorem}{Theorem}[section]
\newtheorem{remark}{Remark}[section]
\newtheorem{lemma}{Lemma}[section]
\newtheorem{definition}{Definition}[section]
\newtheorem{proposition}{Proposition}[section]
\newtheorem{example}{Example}[section]

\begin{document}

\title{THURSTON'S OPERATIONS ON THE BRAID GROUPS}

\author{VIKTOR  LOPATKIN  \footnote{wickktor@gmail.com}}

\maketitle
\abstract{In this paper we study Thurston's automaton on the braid groups via binary operations. These binary operations are obtained from the construction of this automaton. We study these operations and find some connections between them in a ``skew lattice'' spirit.}

\section*{Introduction}
The aim of this paper is to investigate the working of Thurston's automaton. Thurston constructed a finite state automaton \cite[Chapter 9]{EpThur}, having as the set of states the positive non-repeating braids, i.e., any two of its strands cross at most once. The concept of non-repeating braids is very useful, because the total (algebraic) number of crossings of two given strands in a braid is clearly an invariant of isotopy. Since a positive braid has only positive crossing, the absolute number of crossings of two strands in a positive braid is an invariant of isotopy. This idea is very useful, because we can forget about isotopic equivalence and, moreover, there is a bijection between the set of non-repeating braids and permutations.

\par An interesting characteristic of Thurston's automaton is that, after a word is imputed, the state is the maximal tail of the word that lies in the set of non-repeating braids. This automaton allowed proving that the braid group is automatic. Moreover, this automaton rewrites any word into a canonical form which is called the (left or right) greedy normal form.

\par In \cite{Deh1} an original point of view was presented on the working of Thurston's automaton. Namely, there was introduced a concept of  ``derivation''  $\partial w$ of a positive braid $w$. This concept can be described as follows. Usually, we consider a braid $w$ as a three--dimensional figure viewed from the top; the derivation of the braid is the same figure but it viewed from the side. Thus we obtain a new braid which is called the derivative $\partial w$ of $w$. As an application, a normal form was deduced for the positive braid words, which coincided with the ``right greedy normal form''  \cite{Elr}, \cite{EpThur}.

\par As is well known, F. Garside \cite{Gar} solved the Conjugacy Problem for the braid group $\mathbf{Br}_n$ by introducing a submonoid $\mathbf{Br}^+_n$ and a distinguished element $\Delta_n$ of $\mathbf{Br}^+_n$ (however, we denote this element by $\Omega_n$ \cite{EpThur}, because it corresponds to the last (in some sense) permutation $\omega$ which sends $\{1,2, \ldots, n\}$ to $\{n,n-1,\ldots, 1\}$) that he call fundamental, and showing that every element of $\mathbf{Br}_n$ can be expressed as a fraction of the form $\Omega_n^m w$, with $m$ being an integer and $w \in \mathbf{Br}_n^+$.

\par Although F. Garside was very close to such a decomposition when he proved that the greatest common divisors exist in $\mathbf{Br}_n^+$, the result did not appear in his work explicitly, and it seems that the first instances of such distinguished decompositions, or normal forms, go back to the 1980's, to the independent works by S. Adjan \cite{Adjan}, M. El Rifai and H. Morton \cite{Elr},
and W. Thurston (circulated notes \cite{Thur}, later appearing as Chapter IX of the book \cite{EpThur} by D. Epstein et al.). The normal form was soon used to improve Garside's solution of the Conjugacy Problem \cite{Elr} and, extended from the monoid to the group, to serve as
a paradigmatic example in the then emerging theory of automatic groups due to J. Cannon,
W. Thurston, and others. Sometimes called the {\it greedy normal form} or {\it Garside normal
form}, or {\it Thurston normal form}, it became a standard tool in the investigation of braids
and Artin --- Tits monoids and groups from a viewpoint of geometric group theory and of theory of representations, essential, in particular, in D. Krammer's algebraic proof of the linearity of the braid groups \cite{161} and \cite{162}.

\par In this paper we study the Thurston's automaton via new binary operations $\asymp$ and $\bowtie$. These operations are resulted from the construction of this automaton, that is, the Thurston automaton works in the following way. Suppose we have two non-repeating braids $a$ and $b$, and we want to rewrite the word $ab$. The braid $b$ looks for a new crossing of the braid $a$, and if the braid $a$ allows to take this crossing (i.e., if there is a presentation $a = a'a''$ such that $a''$ is a braid which exactly contains the needed crossing for the braid $b$), then the braid $b$ takes this crossing. So, the operation ``give the needed crossing'' from the braid $a$ to the braid $b$ will be denoted as $a \asymp b$ and the operation ``take the needed crossing'' from the braid $a$ to the braid $b$ will be denoted as  $a\bowtie b$ (see fig.\ref{figCD1}). Roughly speaking, the braid $b$ is hungry and greedy for new crossing every time.

\par We will research these operations via combinatorial way, that is, we will find some very interesting relations between them (see Theorem \ref{GSB}). These relations have a ``skew lattice'' spirit. As a corollary to these relations, we will describe a Gr\"obner --- Shirshov basis for the braid groups and also we present the greedy normal form via these operations (see Theorem\ref{cononform}).

\par Gr\"obner bases and Gr\"obner~--- Shirshov bases were invented independently by A.I. Shirshov for the ideals of free
(commutative, anti-commutative) non-associative algebras \cite{Sh1,Sh2} and free Lie algebras \cite{ShE, Sh2}, by H. Hironaka \cite{Hironaka} for the ideals of the power series algebras (both formal and convergent), and by B. Buchberger \cite{Buchberger} for the ideals of the polynomial algebras.

\section{Braid groups in Thurston's generators}

A {\it braid} is obtained by laying down of parallel pieces of strings and intertwining them, without losing track of the fact that they run essentially in the same direction. If we lay down two braids $B$ and $B'$ in a column, so the end of $B$ matches the beginning of $B'$ strand by strand, we get another braid $BB'$; this operation defines a product in the set of all $n$-strands braids, for a fixed $n>1$. We consider two braids to be equivalent if there is an isotopy between them. The set $\mathbf{Br} = \mathbf{Br}_n$ of isotopy classes of $n$-strand braids has a group structure, because if we concatenate a braid with its mirror image in a horizontal plane, the result is isotopic to the trivial braid (the one with no crossings). We call $\mathbf{Br}_n$ the {\it $n$-strand braid group.}

\smallskip

\par We will use as generators for $\mathbf{Br}_n$ the set of {\it positive crossings}, that is, the crossings between two (necessary adjacent) strands, with the front strand having a positive slope. We denote these generators by $\sigma_1, \ldots, \sigma_{n-1}$. These generators are subject to the following relations:
$$
\begin{cases}
\sigma_i\sigma_j = \sigma_j\sigma_i, \mbox{ if } |i-j| >1,\\
\sigma_i\sigma_{i+1}\sigma_i = \sigma_{i+1}\sigma_i\sigma_{i+1}.
\end{cases}
$$

\smallskip

\begin{remark}[{\bf WARNING!}]\label{warning}
Thurston considered a braid from right to left, i.e., the braid starts on the right, and the crossings get added as we move left. Also, he numbered strands at each horizontal position from the top down. It means that for any two braids $a$ and $b$ the notation $ab$ means that we starts from $b$! In this paper, we will use standard notations, i.e., the product of $ab$ starts from $a$. We will also think of the braids as placed in the vertical direction, and we numerate strands at each vertical position from the left to right. Unless otherwise stated, we will assume that for a fixed braid all its strands are numerated with respect to the top line. And finally, the notation $BB'$ for two braids $B,B' \in \mathbf{Br}$ means that the braid $B$ is above the braid $B'$, i.e, the crossings get added as we move down.

\smallskip

\par {It follows that we should invert all Thurston's formulas!}
\end{remark}

\smallskip

\par One obvious invariant of an isotopy of a braid is the permutation it induces on the order of the strands: given a braid $B$, the strands define a map $p(B)$ from the top set of endpoints to the bottom set of endpoints, which we interpret as a permutation of $\{1, \ldots, n\}$. In this way we get a homomorphism $p:\mathbf{Br}_n \to \mathbb{S}_n$, where $\mathbb{S}_n$ is the symmetric group. The generator $\sigma_i$ is mapped to the transposition $s_i = (i,i+1)$. We denote by $S_n = \{s_1, \ldots, s_{n-1}\}$ the set of generators for the symmetric group $\mathbb{S}_n$.

\par Now we want to define an inverse map $p^{-1}:\mathbb{S}_n \to \mathbf{Br}_n$.  To this end, we need the following definition \cite[p.183]{EpThur}

\smallskip

\begin{definition}\label{conR}
Let $S = \{s_1, \ldots, s_{n-1}\}$ be the set of generators for $\mathbb{S}_n$. Each permutation $\pi$ gives rise to a total order relation $\le_\pi$ on $\{1, \ldots,n\}$  with $i \le_\pi j$ if $\pi(i) < \pi(j)$. We set
$$
R_\pi: = \{(i,j) \in \{1, \ldots, n\} \times \{1,\ldots,n\}|i<j, \, \pi(i) > \pi(j)\}.
$$
\end{definition}

\smallskip

\par The construction gives rise to the following formulas:
\begin{equation}\label{Thurstonformulas1}
R_e = \varnothing, \quad R_{\pi^{-1}} = \pi R_\pi, \quad R_{\pi_1\pi_2} = (\pi_1^{-1}R_{\pi_2})\bigtriangleup R_{\pi_1},
\end{equation}
where $\bigtriangleup$ denotes symmetric difference and the image of a pair under permutation is defined by taking the image of each component and reordering, if necessary, so that the smaller number comes first.

\smallskip

\par Unfortunately, W. Thurston did not prove these formulas, and, since they are important for us, we should prove them.

\smallskip

\begin{lemma}[Thurston's formulas]
For any three permutations $\pi_1$, $\pi_2$ and $\pi$ with $\pi = \pi_1\pi_2$, we have
\[
R_{\pi_1\pi_2} = (\pi_1^{-1}R_{\pi_2})\bigtriangleup R_{\pi_1}, \quad  R_{\pi^{-1}} = \pi R_\pi, \quad R_\varepsilon = \varnothing,
\]
where $\varepsilon$ stands for the identity permutation,  $\bigtriangleup$ denotes symmetric difference and the image of a pair under permutation is defined by taking the image of each component and reordering, if necessary, so the smaller number comes first.
\end{lemma}
\begin{proof}

\par Let us prove the first formula. By definition, we have
\[
R_\pi = \{(i,j): i<j,\, \pi(i) > \pi(j)\} = \{(i,j): i<j,\, \pi_2(\pi_1(i)) > \pi_2(\pi_1(j))\},
\]
we have to consider two cases;

\smallskip

\par 1) Let $\pi_1(i) > \pi_1(j)$ then $(i,j) \in R_{\pi_1}$ and $(\pi_1(i), \pi_1(j)) \notin R_{\pi_2}$, i.e., $(i,j) \notin \pi_1^{-1}R_{\pi_2}$,
\par 2) Let $\pi_1(i) < \pi_1(j)$ then $(i,j) \notin R_{\pi_1}$ and $(\pi_1(i), \pi_1(j)) \in R_{\pi_2}$, i.e., $(i,j) \in \pi_1^{-1}R_{\pi_2}$.

\smallskip

\par It follows from considering these cases that if $(i,j) \in R_\pi$ then $(i,j) \in \left(R_{\pi_1} \cup \pi_1^{-1} R_{\pi_2} \right) \setminus \left( R_{\pi_1} \cap \pi_1^{-1} R_{\pi_2}\right) = R_{\pi_1} \bigtriangleup \pi_1^{-1}R_{\pi_2}$, i.e, we have proved that $R_\pi \subseteq R_{\pi_1} \bigtriangleup \pi_1^{-1}R_{\pi_2}$.

\par Let $(i,j) \in R_{\pi_1} \bigtriangleup \pi_1^{-1}R_{\pi_2} = \left(R_{\pi_1} \setminus \pi_1^{-1}R_{\pi_2} \right) \cup \left( \pi_1^{-1}R_{\pi_2} \setminus R_{\pi_1}\right)$. Assume that $(i,j) \in R_{\pi_1} \setminus \pi_1^{-1}R_{\pi_2}$, i.e., $\pi_1(i) > \pi_1(j)$ and $(i,j) \notin \pi_1^{-1} R_{\pi_2}$ then $(\pi_1(i),\pi_1(j)) \notin R_{\pi_2}$ it follows that $\pi_2(\pi_1(i))> \pi_2(\pi_1(j))$, but it means that for the pair $(i,j)$ with $i<j$ we have $\pi_2(\pi_1(i))> \pi_2(\pi_1(j))$, i.e, $(i,j) \in R_\pi.$ And finally, let us assume that $(i,j) \in \pi_1^{-1}R_{\pi_2} \setminus R_{\pi_1}$, i.e., since $\pi_1(i) < \pi_1(j)$ then it follows that $\pi_2(\pi_1(i))>\pi_2(\pi_1(j))$, it exactly means that $(i,j) \in R_\pi.$ We have just proved that $R_{\pi_1} \bigtriangleup \pi_1^{-1}R_{\pi_2} \subseteq R_\pi$, i.e., $R_\pi = R_{\pi_1} \bigtriangleup \pi_1^{-1}R_{\pi_2}$, as claimed.

\smallskip

\par It is obvious that the identity permutation $\varepsilon$ gives rise to the same total order relation $<$ on $\{1, \ldots, n\}$, i.e., $R_\varepsilon = \varnothing$. Let $\pi_1 = \tau^{-1}$ and $\pi_2 = \tau$, then $\pi_1\pi_2 = \varepsilon$, and, using first formula, we get
\[
\varnothing = R_{\tau^{-1}\tau} =\tau R_{\tau} \bigtriangleup R_{\tau^{-1}} = \left(\tau R_\tau \setminus R_{\tau^{-1}}\right) \cup \left(R_{\tau^{-1}} \setminus \tau R_\tau \right),
\]
but it is possible iff $R_{\tau^{-1}} = \tau R_\tau$, as claimed. The proof is completed.
\end{proof}

\smallskip

\begin{definition}{\cite[\S 9.2]{EpThur}}
For any word $w$ (from a monoid $W$) there is a concept of reversal $w^*$  of the word which is an involution (it is also anti-automorphism of $W$). For any non-repeating braid $R_\pi$ we can also extend this concept. We can define (see \cite[p.191]{EpThur}) an involution for any Thurston's generators in the following way:
\[
R^*_\pi:=\pi R_\pi= R_{\pi^{-1}}.
\]
\end{definition}

\smallskip

\begin{lemma}\cite[Lemma 9.1.6]{EpThur}\label{criteria}
A set $R$ of pairs $(i,j)$, with $i <j$, comes from some permutation if and only if the following two conditions are satisfied:
\par \textup{i)} If $(i,j) \in R$ and $(j,k) \in R$, then $(i,k) \in R$.
\par \textup{ii)} If $(i,k) \in R$, then $(i,j) \in R$ or $(j,k) \in R$ for every $j$ with $i<j<k$.
\end{lemma}

\smallskip

\par Now we will define (\cite[p. 186]{EpThur}) a very important concept of non-repeating braid.

\begin{definition}
Recall that our set of generators $S_n$ includes only positive crossings; the positive braid monoid is denoted by $\mathbf{Br}^+  =\mathbf{Br}_n^+$. We call a positive braid non-repeating if any two of its strands cross at most once. We define $D =D_n \subset \mathbf{Br}^+_n$ as the set of classes of non-repeating braids.
\end{definition}

\smallskip

\par Let us recall some other concepts from \cite{EpThur}. A partial order in $\mathbb{S}_n$ is defined by setting $\pi_1 \le \pi_2$ if $R_{\pi_1} \subset R_{\pi_2}$. Then, the identity $\varepsilon = \begin{pmatrix}1& \ldots& n\\ 1& \ldots& n \end{pmatrix}$ is the smallest element of $\mathbb{S}_n$ with respect to $\ge$. The largest element is the permutation $\begin{pmatrix}1& \ldots& n\\ n& \ldots& 1 \end{pmatrix}$, which we denote by $\omega$. The corresponding braid $R_\omega$ will be denoted by $\Omega = \Omega_n$ (for more information about this braid (Garside's braid) see below). It is not hard to see that all strands in $\Omega_n$ are crossed.

\smallskip

\par Further, the equation (\ref{Thurstonformulas1}) shows that, if $\pi$ is a permutation, then
\[
R_{\pi \omega} = (\pi^{-1}\Omega) \bigtriangleup R_{\pi} = \Omega \bigtriangleup R_{\pi} = \Omega \setminus R_\pi,
\]
we used the fact that all strands in $\Omega$ are crossed. It follows that for any permutation $\tau$ we have $\tau \Omega = \Omega$.

\begin{definition}\label{star}
A complementation operation $\neg: D \to D$ is defined in the following way
\[
\neg R_\pi : = R_{\pi\omega}.
\]
\end{definition}

\smallskip

\par The following lemma summarizes all the above mentioned concepts and notations.

\smallskip

\begin{lemma}\cite[Lemma 9.1.10 and Lemma 9.1.11]{EpThur}\label{9.1.11}
The homomorphism $p:\mathbf{Br}_n^+ \to \mathbb{S}_n$ is restricted to a bijection $D \to \mathbb{S}_n$. A positive braid $B$ is non-repeating iff $|B| = |p(B)|$ (here $|?|$ means the length of a word). If a non-repeating braid maps to a permutation $\pi$, two strands $i$ and $j$ cross iff $(i,j) \in R_\pi$. Further, if $R_a, R_b \in D$, we have $R_a\cdot R_b = R_{ab}\in D$ iff $\neg R^*_a \supseteq R_b$.
\end{lemma}

\smallskip

\par Let us illustrate all these concepts and definitions via the following

\begin{example}
Let us consider the permutation
\[
\pi = \begin{pmatrix}1&2&3&4&5&6 \\ 4&2&6&1&5&3 \end{pmatrix},
\]
we have
\begin{align*}
& \begin{cases}1<2,\\ \pi(1) > \pi(2)\end{cases}, \quad \begin{cases}1<4, \\ \pi(1) > \pi(4)\end{cases}, \quad  \begin{cases}1<6, \\ \pi(1) > \pi(6)\end{cases}, \quad  \begin{cases}2<4, \\ \pi(2) > \pi(4)\end{cases},\\
& \begin{cases}3<4,\\ \pi(3) > \pi(4)\end{cases}, \quad \begin{cases}3<5,\\ \pi(3) > \pi(5)\end{cases}, \quad \begin{cases}3<6,\\ \pi(3) > \pi(6)\end{cases}, \quad \begin{cases}5<6,\\ \pi(5) > \pi(6)\end{cases},
\end{align*}
it follows that
\[
R_\pi =  \{(1,2),(1,4),(1,6),(2,4),(3,4),(3,5),(3,6),(5,6)\}.
\]

\smallskip

\par The corresponding braid is shown in the figure \ref{dem1}. We can also find the permutation $\neg \pi = \pi \omega$; we have
\[
\neg \pi = \pi \omega = \begin{pmatrix}1&2&3&4&5&6 \\ 4&2&6&1&5&3 \end{pmatrix} \begin{pmatrix}1&2&3&4&5&6 \\ 6&5&4&3&2&1 \end{pmatrix} = \begin{pmatrix}1&2&3&4&5&6 \\ 3&5&1&6&2&4 \end{pmatrix},
\]

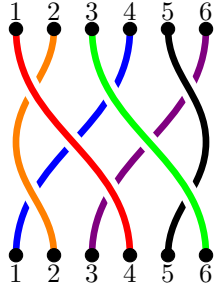
\begin{figure}[h!]
\begin{center}
\begin{tikzpicture}
%6
\draw [line width = 2.5,violet] (2.5,0) to [out = 270, in = 90] (1,-3);
%5
\draw [line width = 7,white] (2,0) to [out = 270, in = 90] (2.5,-1.5) to [out = 270, in = 90] (2,-3);
\draw [line width = 2.5,black] (2,0) to [out = 270, in = 90] (2.5,-1.5) to [out = 270, in = 90] (2,-3);
%4
\draw [line width = 7,white] (1.5,0) to [out = 270, in = 90] (0,-3);
\draw [line width = 2.5,blue] (1.5,0) to [out = 270, in = 90] (0,-3);
%3
\draw [line width = 7,white] (1,0) to [out = 270, in = 90] (2.5,-3);
\draw [line width = 2.5,green] (1,0) to [out = 270, in = 90] (2.5,-3);
%2
\draw [line width = 7,white] (0.5,0) to [out = 270, in = 90] (0,-1.5) to [out = 270, in = 90] (0.5,-3);
\draw [line width = 2.5,orange] (0.5,0) to [out = 270, in = 90] (0,-1.5) to [out = 270, in = 90] (0.5,-3);
%1
\draw [line width = 7,white] (0,0) to [out = 270, in = 90] (1.5,-3);
\draw [line width = 2.5,red] (0,0) to [out = 270, in = 90] (1.5,-3);
%points
%-----------------------------------
\draw [fill](0,0) circle (2.5pt);
\draw [fill](0.5,0) circle (2.5pt);
\draw [fill](1,0) circle (2.5pt);
\draw [fill](1.5,0) circle (2.5pt);
\draw [fill](2,0) circle (2.5pt);
\draw [fill](2.5,0) circle (2.5pt);
%-----------------------------------
\draw [fill](0,-3) circle (2.5pt);
\draw [fill](0.5,-3) circle (2.5pt);
\draw [fill](1,-3) circle (2.5pt);
\draw [fill](1.5,-3) circle (2.5pt);
\draw [fill](2,-3) circle (2.5pt);
\draw [fill](2.5,-3) circle (2.5pt);
%-----------------------------------
%.................................................
%numbers
\node[above] at (0,0){$1$};
\node[above] at (0.5,0){$2$};
\node[above] at (1,0){$3$};
\node[above] at (1.5,0){$4$};
\node[above] at (2,0){$5$};
\node[above] at (2.5,0){$6$};
%-------------------------------
\node[below] at (0,-3){$1$};
\node[below] at (0.5,-3){$2$};
\node[below] at (1,-3){$3$};
\node[below] at (1.5,-3){$4$};
\node[below] at (2,-3){$5$};
\node[below] at (2.5,-3){$6$};
\end{tikzpicture}
\caption{Here a non-repeating braid is shown.}\label{dem1}
\end{center}
\end{figure}

we have
\[
\neg R_\pi = R_{\neg \pi} = \{(1,3),(1,5),(2,3),(2,5),(2,6),(4,5),(4,6)\},
\]
we see that these pairs correspond exactly to the non-crossing strands. Let us consider the action of the permutation $\pi$ over the set $R_\pi$, i.e., let us consider the images of all pairs of $R_\pi$ under the permutation $\pi$. As was explained, the image of a pair under permutation is defined by taking the image of each component and reordering, if necessary, so the smaller number comes first. We have
\begin{multline*}
\pi R_\pi = \{(\pi(1),\pi(2)),(\pi(1),\pi(4)),(\pi(1),\pi(6)),(\pi(2),\pi(4)),(\pi(3),\pi(4)),(\pi(3),\pi(5)),
(\pi(3),\pi(6)),(\pi(5),\pi(6))\} = \\ = \{(2,4),(1,4),(3,4),(1,2),(1,6),(5,6),(3,6),(3,5)\} = \{(1,2),(1,4),(1,6),(2,4),(3,4),(3,5),(3,6),(5,6)\},
\end{multline*}
we see that these pairs are exactly the crossing strands which are numbered with respect to the bottom boundary of the braid.
\end{example}

\section{Thurston's operations and the greedy normal form}
In this section we will discuss and study the working of Thurston's automaton; this automaton allows defining some binary operations on non-repeating braids. We will describe these operations and prove some very interesting formulas; as a result, we will obtain the Gr\"obner --- Shirshov basis for the braid monoids.

\smallskip

\begin{definition}\cite[Proposition 9.1.8]{EpThur} The partial order $\ge$ imposes a lattice structure on $\mathbb{S}_n$, that is, given permutations $\pi_1$ and $\pi_2$, there is a largest element $\pi_1 \wedge \pi_2$ smaller than $\pi_1$ and $\pi_2$, which is defined by the set
\[
R_{\pi_1} \wedge R_{\pi_2} :=\{(i,k)\in R_{\pi_1} \cap R_{\pi_2}: (i,j) \in R_{\pi_1} \cap R_{\pi_2} \mbox{ or } (j,k) \in R_{\pi_1} \cap R_{\pi_2} \mbox{ for all $j$ with } i < j < k\},
\]
there is a smallest element $\pi_1 \vee \pi_2$ larger than $\pi_1$ and $\pi_2$, which can be defined as
\[
R_{\pi_1} \vee R_{\pi_2}: = \neg(\neg R_{\pi_1} \wedge \neg R_{\pi_2}).
\]
\end{definition}

\par Roughly speaking, the set $R_{\pi_1} \cap R_{\pi_2}$ should satisfy the condition ii) of Lemma \ref{criteria}, otherwise we have to put $R_{\pi_1} \wedge R_{\pi_2} = R_\varepsilon$. The following example can help to understand this concept (see also \cite[p.185]{EpThur}).

\smallskip

\begin{example}
Let us consider the following braids $R_a$ and $R_b$ (see fig.\ref{GNF}), where
\[
a = \begin{pmatrix}1&2&3&4&5&6\\3&5&4&2&6&1 \end{pmatrix}, \quad b = \begin{pmatrix}1&2&3&4&5&6\\2&1&5&6&3&4 \end{pmatrix},
\]
\begin{figure}[h!]
\begin{center}
\begin{tikzpicture}
%111111111111111111111111111111111111111
%6
\draw [line width = 2.5,red] (2.5,0) to [out = 270, in = 90] (0,-3);
%5
\draw [line width = 7,white] (2,0) to [out = 270, in = 90] (2.5,-3);
\draw [line width = 2.5,red] (2,0) to [out = 270, in = 90] (2.5,-3);
%4
\draw [line width = 7,white] (1.5,0) to [out = 270, in = 90] (0,-1.5) to [out = 270,in = 90] (0.5,-3);
\draw [line width = 2.5,red] (1.5,0) to [out = 270, in = 90] (0,-1.5) to [out = 270,in = 90] (0.5,-3);
%3
\draw [line width = 7,white] (1,0) to [out = 270, in = 90] (1.7,-1.5) to [out = 270,in = 90] (1.5,-3);
\draw [line width = 2.5,red] (1,0) to [out = 270, in = 90] (1.7,-1.5) to [out = 270,in = 90] (1.5,-3);
%2
\draw [line width = 7,white] (0.5,0) to [out = 270, in = 90] (2,-3);
\draw [line width = 2.5,red] (0.5,0) to [out = 270, in = 90] (2,-3);
%1
\draw [line width = 7,white] (0,0) to [out = 270, in = 90] (1,-3);
\draw [line width = 2.5,red] (0,0) to [out = 270, in = 90] (1,-3);
%22222222222222222222222222222222222222222222222222222222222222222222222
%6
\draw [line width = 2.5,blue] (2.5,-3) to [out = 270, in = 90] (1.5,-6);
%5
\draw [line width = 2.5,blue] (2,-3) to [out = 270, in = 90] (1,-6);
%4
\draw [line width = 7,white] (1.5,-3) to [out = 270, in = 90] (2.5,-6);
\draw [line width = 2.5,blue] (1.5,-3) to [out = 270, in = 90] (2.5,-6);
%3
\draw [line width = 7,white] (1,-3) to [out = 270, in = 90] (2,-6);
\draw [line width = 2.5,blue] (1,-3) to [out = 270, in = 90] (2,-6);
%2
\draw [line width = 2.5,blue] (0.5,-3) to [out = 270, in = 90] (0,-6);
%1
\draw [line width = 7,white] (0,-3) to [out = 270, in = 90] (0.5,-6);
\draw [line width = 2.5,blue] (0,-3) to [out = 270, in = 90] (0.5,-6);
%points
%-----------------------------------
\draw [fill](0,0) circle (2.5pt);
\draw [fill](0.5,0) circle (2.5pt);
\draw [fill](1,0) circle (2.5pt);
\draw [fill](1.5,0) circle (2.5pt);
\draw [fill](2,0) circle (2.5pt);
\draw [fill](2.5,0) circle (2.5pt);
%-----------------------------------
\draw [fill](0,-3) circle (2.5pt);
\draw [fill](0.5,-3) circle (2.5pt);
\draw [fill](1,-3) circle (2.5pt);
\draw [fill](1.5,-3) circle (2.5pt);
\draw [fill](2,-3) circle (2.5pt);
\draw [fill](2.5,-3) circle (2.5pt);
%-----------------------------------
\draw [fill](0,-6) circle (2.5pt);
\draw [fill](0.5,-6) circle (2.5pt);
\draw [fill](1,-6) circle (2.5pt);
\draw [fill](1.5,-6) circle (2.5pt);
\draw [fill](2,-6) circle (2.5pt);
\draw [fill](2.5,-6) circle (2.5pt);
%...............................................................................
%numbers
\node[above] at (0,0){$1$};
\node[above] at (0.5,0){$2$};
\node[above] at (1,0){$3$};
\node[above] at (1.5,0){$4$};
\node[above] at (2,0){$5$};
\node[above] at (2.5,0){$6$};
%-------------------------------
\node[below] at (0,-6){$1$};
\node[below] at (0.5,-6){$2$};
\node[below] at (1,-6){$3$};
\node[below] at (1.5,-6){$4$};
\node[below] at (2,-6){$5$};
\node[below] at (2.5,-6){$6$};
\end{tikzpicture}
\end{center}
\caption{Despite the fact that $aR_a \cap \neg R_b= \{(1,3), (1,4), (1,5), (1,6), (2,3), (2,4), (2,5)\} \ne \varnothing$, nevertheless, the braid criterion is false, i.e., $aR_a \wedge \neg R_b = R_\varepsilon.$}\label{GNF}
\end{figure}
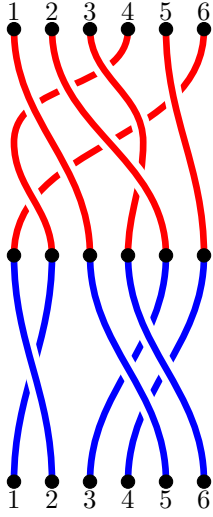
then we get
\begin{align*}
&aR_a = \{(1,2), (1,3), (1,4), (1,5), (1,6), (2,3), (2,4), (2,5), (4,5)\},\\
&\neg R_b = \{(1,3), (1,4), (1,5),(1,6), (2,3), (2,4), (2,5), (2,6), (3,4), (5,6)\},
\end{align*}
then
\[
aR_a \cap \neg R_b = \{(1,3), (1,4), (1,5), (1,6), (2,3), (2,4), (2,5)\},
\]
using Lemma \ref{criteria}, we see that $(1,2) \notin aR_a \cap \neg R_b$ or $(2,6) \notin aR_a \cap \neg R_b$ then $aR_a \wedge \neg R_b = R_\varepsilon$.
\end{example}

\smallskip

\par Let us consider two partial orderings \cite[Definition 9.1.12]{EpThur}  in the semigroup $\mathbf{Br}^+$ , defined as follows: if $ab = c$ for $a,b,c \in \mathbf{Br}^+$, we write $a \prec c$ and $c \succ b$, and say that $a$ is a {\it head} and $b$ is a {\it tail} of $c$. Now we can present Thurston's automaton $M$ over $A$, having the set of states $D$. A characteristic feature of $M$ is that, after a word $b$ is inputted, the state $d$ of $M$ is the maximal tail of $R_b \in \mathbf{Br}^+$ that lies in $D$.

\smallskip

\begin{definition}{\cite[Proposition 9.2.1]{EpThur}}\label{ThurA} If $M$ is in a state $a \in A$ and we input a word $b\in A$ representing an element $R_b\in D$, then the resulting state is
\[
M(a,w) = (aR_a \wedge \neg R_w)R_w;
\]
the element $aR_a \wedge \neg R_b \in D$ is the maximal tail of $R_a$ that gives an element of $D$ when multiplied on the right by $R_b$.
\end{definition}

\smallskip

\begin{remark}
As we mentioned (see remark \ref{warning}), Thurston considered that a braid starts on the right, and that the crossings get added as we move left, i.e., $R_aR_b$ means that $R_b$ is the first braid, and then we have to add the braid $R_a$ on the left. In his original definition (see \cite[Proposition 9.2.1]{EpThur}) he defined the automaton $M$ via the following formula
\[
M(a,b) = (R_a \wedge \neg bR_b)R_b,
\]
but since we use standard notations, this formula has to be rewritten in the above mentioned form.
\end{remark}

\smallskip

\par From the construction of Thurston's automaton $M$, it follows that $M$ finds maximal tails and, as a result, it rewrites any word to some form. We have the following

\smallskip

\begin{definition}{\cite[Theorem 9.2.2, Propositon 9.2.3]{EpThur}} A word $R_w$ over $S_n$ is in the right--greedy normal form iff it has a decomposition
\[
R_w = R_{\pi_1}R_{\pi_2} \cdots R_{\pi_\ell},
\]
where each $R_{\pi_i} \in D$ is a non-repeating braid and $\pi_iR_{\pi_i} \wedge \neg R_{\pi_{i+1}} = R_\varepsilon$, for any $1 \le i < \ell$. Geometrically, if two strands that are adjacent at the boundary of $R_{\pi_i}$ and $R_{\pi_{i+1}}$ cross in $R_{\pi_i}$, they also cross in $R_{\pi_{i+1}}$.
\end{definition}

\par For example, let us consider the fig.\ref{GNF}. We have seen that $aR_a \wedge \neg R_b = R_\varepsilon$, so $R_aR_b$ is the right greedy normal form. Also, if we have a look at the boundary of these braids, then we see that the first and second strands are crossed in the red and blue braids, it is also true for the forth and fifth strands.

\smallskip

\par The construction of Thurston's automaton gives rise the following

\begin{definition}\label{con1}
Let $R_a$ and $R_b$ be non-repeating braids, let us define two non-repeating braids $R_a \asymp R_b$ and $R_a \bowtie R_b$ by the following formulas:
\begin{align}
&R_a \asymp R_b : = R_{a\left( a^{-1} \wedge \omega b \right)},\label{a-b}\\
&R_a \bowtie R_b: = R_{\left( a^{-1} \wedge \omega b\right)^{-1}b}.\label{a+b}
\end{align}
\end{definition}

\smallskip

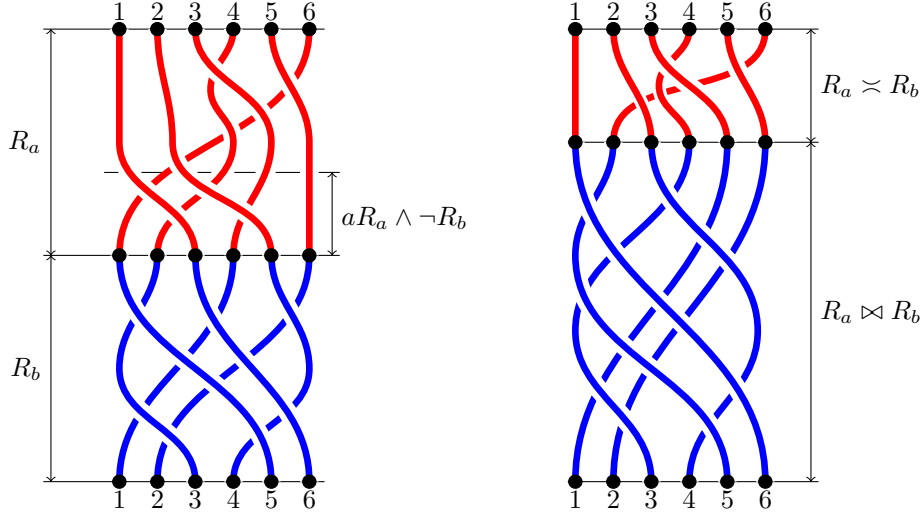
\begin{figure}
\begin{center}
\begin{tikzpicture}
%111111111111111111111111111111111111111111111111111111111111111
%lines
%-------------------
\draw[thin] (2.7,0) -- (-1,0);
\draw[thin] (-0.2,-1.9) -- (2.9,-1.9);
\draw[thin] (2.9,-3) -- (-1,-3);
\draw[thin] (2.7,-6) -- (-1,-6);
\draw[<->,thin] (-0.9,0) -- (-0.9,-3);
\draw[<->,thin] (-0.9,-3) -- (-0.9,-6);
\draw[<->,thin] (2.8,-1.9) -- (2.8,-3);
\node[right] at (2.8,-2.5){$aR_a\wedge \neg R_b$};
\node[left] at (-0.9,-1.5){$R_a$};
\node[left] at (-0.9,-4.5){$R_b$};
%-------------------
%sixth
\draw [line width = 2.5,red] (2.5,0) to [out = 270, in = 90] (0,-3);
%fifth
\draw [line width = 7,white] (2,0) to [out = 270, in = 90] (2.5,-1.5) to [out = 270, in = 90] (2.5,-3);
\draw [line width = 2.5,red] (2,0) to [out = 270, in = 90] (2.5,-1.5) to [out = 270, in = 90] (2.5,-3);
%fourth
\draw [line width = 7,white] (1.5,0) to [out = 270, in = 90] (1.2,-0.8) to [out = 270,in = 90] (1.5,-1.5) to [out = 270, in = 90] (0.5,-3);
\draw [line width = 2.5,red] (1.5,0) to [out = 270, in = 90] (1.2,-0.8) to [out = 270,in = 90] (1.5,-1.5) to [out = 270, in = 90] (0.5,-3);
%third
\draw [line width = 7,white] (1,0) to [out = 270, in = 90] (2,-1.5) to [out = 270, in = 90] (1.5,-3);
\draw [line width = 2.5,red] (1,0) to [out = 270, in = 90] (2,-1.5) to [out = 270, in = 90] (1.5,-3);
%second
\draw [line width = 7,white] (0.5,0) to [out = 270, in = 90] (0.7,-1.5) to [out = 270, in = 90] (2,-3);
\draw [line width = 2.5,red] (0.5,0) to [out = 270, in = 90] (0.7,-1.5) to [out = 270, in = 90] (2,-3);
%first
\draw [line width = 7,white] (0,0) to [out = 270, in = 90] (0,-1.5) to [out = 270, in = 90] (1,-3);
\draw [line width = 2.5,red] (0,0) to [out = 270, in = 90] (0,-1.5) to [out = 270, in = 90] (1,-3);
%-----------------------------------------------
%button braid
%sixth
\draw [line width = 2.5,blue] (2.5,-3) to [out = 270, in=90] (0.5,-6);
%fifth
\draw [line width = 7,white] (2,-3) to [out = 270, in = 90] (2.5,-4.5) to [out = 270, in=90] (1.5,-6);
\draw [line width = 2.5,blue] (2,-3) to [out = 270, in = 90] (2.5,-4.5) to [out = 270, in=90] (1.5,-6);
%fourth
\draw [line width = 2.5,blue] (1.5,-3) to [out = 270, in = 90] (0,-6);
%third
\draw [line width = 7,white] (1,-3) to [out = 270, in = 90] (2.5,-6);
\draw [line width = 2.5,blue] (1,-3) to [out = 270, in = 90] (2.5,-6);
%second
\draw [line width = 7,white] (0.5,-3) to [out = 270, in = 90] (0,-4.5) to [out=270, in = 90] (1,-6);
\draw [line width = 2.5,blue] (0.5,-3) to [out = 270, in = 90] (0,-4.5) to [out=270, in = 90] (1,-6);
%first
\draw [line width = 7,white] (0,-3) to [out = 270, in = 90] (2,-6);
\draw [line width = 2.5,blue] (0,-3) to [out = 270, in = 90] (2,-6);
%points
%-----------------------------------
\draw [fill](0,0) circle (2.5pt);
\draw [fill](0.5,0) circle (2.5pt);
\draw [fill](1,0) circle (2.5pt);
\draw [fill](1.5,0) circle (2.5pt);
\draw [fill](2,0) circle (2.5pt);
\draw [fill](2.5,0) circle (2.5pt);
%-----------------------------------
\draw [fill](0,-3) circle (2.5pt);
\draw [fill](0.5,-3) circle (2.5pt);
\draw [fill](1,-3) circle (2.5pt);
\draw [fill](1.5,-3) circle (2.5pt);
\draw [fill](2,-3) circle (2.5pt);
\draw [fill](2.5,-3) circle (2.5pt);
%-----------------------------------
\draw [fill](0,-6) circle (2.5pt);
\draw [fill](0.5,-6) circle (2.5pt);
\draw [fill](1,-6) circle (2.5pt);
\draw [fill](1.5,-6) circle (2.5pt);
\draw [fill](2,-6) circle (2.5pt);
\draw [fill](2.5,-6) circle (2.5pt);
%...............................................................................
%numbers
\node[above] at (0,0){$1$};
\node[above] at (0.5,0){$2$};
\node[above] at (1,0){$3$};
\node[above] at (1.5,0){$4$};
\node[above] at (2,0){$5$};
\node[above] at (2.5,0){$6$};
%-------------------------------
\node[below] at (0,-6){$1$};
\node[below] at (0.5,-6){$2$};
\node[below] at (1,-6){$3$};
\node[below] at (1.5,-6){$4$};
\node[below] at (2,-6){$5$};
\node[below] at (2.5,-6){$6$};
%222222222222222222222222222222222222222222222222222222222222222222
\draw[thin] (5.9,0) -- (9.2,0);
\draw[thin] (5.9,-1.5) -- (9.2,-1.5);
\draw[thin] (5.9,-6) -- (9.2,-6);
\draw[<->,thin] (9.1,0) -- (9.1,-1.5);
\draw[<->,thin] (9.1,-1.5) -- (9.1,-6);
\node[right] at (9.1,-0.75){$R_a \asymp R_b$};
\node[right] at (9.1,-3.75){$R_a \bowtie R_b$};
%-------------------------------------------------------------------
\draw [line width = 2.5,red] (8.5,0) to [out = 270, in = 90] (6.5,-1.5);
\draw [line width = 7,white] (8,0) to [out = 270, in = 90] (8.5,-1.5);
\draw [line width = 2.5,red] (8,0) to [out = 270, in = 90] (8.5,-1.5);
\draw [line width = 5,white] (7.5,0) to [out = 270, in = 90] (7.1,-0.75) to [out = 270,in = 90] (7.5,-1.5);
\draw [line width = 2.5,red] (7.5,0) to [out = 270, in = 90] (7.1,-0.75) to [out = 270,in = 90] (7.5,-1.5);
\draw [line width = 7,white] (7,0) to [out = 270, in = 90] (8,-1.5);
\draw [line width = 2.5,red] (7,0) to [out = 270, in = 90] (8,-1.5);
\draw [line width = 5,white] (6.5,0) to [out = 270, in = 90] (7,-1.5);
\draw [line width = 2.5,red] (6.5,0) to [out = 270, in = 90] (7,-1.5);
\draw [line width = 2.5,red] (6,0) to [out = 270, in = 90] (6,-1.5);
%----------------------------------------------------------------------
%6
\draw [line width = 2.5,blue] (8.5,-1.5) to [out = 270, in = 90] (6.5,-6);
%5
\draw [line width = 2.5,blue] (8,-1.5) to [out = 270, in = 90] (6,-6);
%4
\draw [line width = 7,white] (7.5,-1.5) to [out = 270, in = 90] (6,-4) to [out = 270,in=90] (7,-6);
\draw [line width = 2.5,blue] (7.5,-1.5) to [out = 270, in = 90] (6,-4) to [out = 270,in=90] (7,-6);
%3
\draw [line width = 7,white] (7,-1.5) to [out = 270, in = 90] (8.4,-4) to [out = 270,in=90] (7.5,-6);
\draw [line width = 2.5,blue] (7,-1.5) to [out = 270, in = 90] (8.4,-4) to [out = 270,in=90] (7.5,-6);
%2
\draw [line width = 7,white] (6.5,-1.5) to [out = 270, in = 90] (6,-3) to [out = 270,in=90] (8,-6);
\draw [line width = 2.5,blue] (6.5,-1.5) to [out = 270, in = 90] (6,-3) to [out = 270,in=90] (8,-6);
%1
\draw [line width = 7,white] (6,-1.5) to [out = 270, in = 90] (8.5,-6);
\draw [line width = 2.5,blue] (6,-1.5) to [out = 270, in = 90] (8.5,-6);
%----------------------------------
%points
%-----------------------------------
\draw [fill](6,0) circle (2.5pt);
\draw [fill](6.5,0) circle (2.5pt);
\draw [fill](7,0) circle (2.5pt);
\draw [fill](7.5,0) circle (2.5pt);
\draw [fill](8,0) circle (2.5pt);
\draw [fill](8.5,0) circle (2.5pt);
%-----------------------------------
\draw [fill](6,-1.5) circle (2.5pt);
\draw [fill](6.5,-1.5) circle (2.5pt);
\draw [fill](7,-1.5) circle (2.5pt);
\draw [fill](7.5,-1.5) circle (2.5pt);
\draw [fill](8,-1.5) circle (2.5pt);
\draw [fill](8.5,-1.5) circle (2.5pt);
%-----------------------------------
\draw [fill](6,-6) circle (2.5pt);
\draw [fill](6.5,-6) circle (2.5pt);
\draw [fill](7,-6) circle (2.5pt);
\draw [fill](7.5,-6) circle (2.5pt);
\draw [fill](8,-6) circle (2.5pt);
\draw [fill](8.5,-6) circle (2.5pt);
%----------------------------------------
%numbers
\node[above] at (6,0){$1$};
\node[above] at (6.5,0){$2$};
\node[above] at (7,0){$3$};
\node[above] at (7.5,0){$4$};
\node[above] at (8,0){$5$};
\node[above] at (8.5,0){$6$};
%-------------------------------
\node[below] at (6,-6){$1$};
\node[below] at (6.5,-6){$2$};
\node[below] at (7,-6){$3$};
\node[below] at (7.5,-6){$4$};
\node[below] at (8,-6){$5$};
\node[below] at (8.5,-6){$6$};
\end{tikzpicture}
\end{center}
\caption{Here is shown the rewriting $R_aR_b$ (where $R_a$ is shown as the red braid and $R_b$ is shown as the blue braid) to the right greedy normal form. The rewriting process looks like addition of the maximal tail $aR_a\wedge \neg R_b$ of the red braid to the blue braid.}\label{figCD1}
\end{figure}
\begin{example}
Let us consider the following braids (see fig.\ref{figCD1}). We have
$$
R_a =\{(1,4),(1,6),(2,3),(2,4),(2,6),(3,4),(3,6),(5,6)\}, \quad a= \begin{pmatrix} 1&2&3&4&5&6 \\ 3&5&4&2&6&1  \end{pmatrix},
$$

$$
 R_a^* = aR_a = \{(1,2),(1,3),(1,4),(1,5),(1,6),(2,3),(2,4),(2,5),(4,5)\} , \quad  a^{-1} = \begin{pmatrix} 1 & 2&3&4&5&6 \\ 6&4&1&3&2&5  \end{pmatrix}
$$
$$
\neg R_b = \{(1,3),(2,3),(2,5),(4,5),(4,6)\}, \, b = \begin{pmatrix}1&2&3&4&5&6\\ 5&3&6&1&4&2 \end{pmatrix}, \, \omega b = \begin{pmatrix}1&2&3&4&5&6\\2&4&1&6&3&5  \end{pmatrix}
$$
then we have
$$
R_a^* \cap \neg R_b = \{(1,3), (2,3),(2,5),(4,5)\},
$$
It is not difficult to see that this set satisfies the conditions of Lemma \ref{criteria}, i.e., we have
$$
R_a^* \wedge \neg R_b = \{(1,3), (2,3),(2,5),(4,5)\},\qquad  a^{-1} \wedge b \omega= \begin{pmatrix}1&2&3&4&5&6\\2&4&1&5&3&6 \end{pmatrix}.
$$
Thus we get
$$
a (a^{-1} \wedge b\omega ) = \begin{pmatrix} 1&2&3&4&5&6 \\ 3&5&4&2&6&1  \end{pmatrix} \begin{pmatrix}1&2&3&4&5&6\\2&4&1&5&3&6 \end{pmatrix}=  \begin{pmatrix}1&2&3&4&5&6\\ 1&3&5&4&6&2  \end{pmatrix}
$$
and
$$
(a^{-1} \wedge b\omega )^{-1}b = \begin{pmatrix}1&2&3&4&5&6\\3&1&5&2&4&6 \end{pmatrix} \begin{pmatrix}1&2&3&4&5&6\\ 5&3&6&1&4&2 \end{pmatrix} = \begin{pmatrix}1&2&3&4&5&6\\ 6&5&4&3&1&2 \end{pmatrix}.
$$
\end{example}

\smallskip

\par Thurston's operation $\asymp$ and $\bowtie$ has another interpretation via a set theory spirit (see the proposition below), but we will see that this interpretation is not so useful for our purposes.

\smallskip

\begin{proposition}
For any two non-repeating braids $R_a, R_b \in D$ with $aR_a \wedge \neg R_b \ne R_\varepsilon$, the Thurston's operations can be described in the following way
\begin{align*}
& R_a \asymp R_b = R_a \cap a^{-1}R_b,\\
& (R_a \bowtie R_b)^* = bR_b \cup baR_a,
\end{align*}
the sets $R_a \cap a^{-1}R_b$, $bR_b \cup baR_a$ satisfy to the conditions of Lemma \ref{criteria}.
\end{proposition}
\begin{proof}
Since the braid $R_x = aR_a \wedge \neg R_b$ is not trivial, then, using (\ref{a-b}) and (\ref{Thurstonformulas1}), we get
\[
R_a \asymp R_b = R_{ax} = (a^{-1}R_x)\bigtriangleup R_a,
\]
but since $R_x \subseteq aR_a$ then $a^{-1}R_x \subseteq R_a$ it follows that
\begin{multline*}
R_a \asymp R_b = R_{ax} = (a^{-1}R_x)\bigtriangleup R_a = R_a \setminus a^{-1}R_x = R_a \cap a^{-1}(\neg R_x) = R_a \cap a^{-1} \left(\neg \left(aR_a \cap \neg R_b \right) \right) =\\= R_a \cap a^{-1}\left(a \neg R_a \cup aR_b \right) = (R_a \cap \neg R_a) \cup (R_a \cap aR_b) = R_a \cap a^{-1}R_b,
\end{multline*}
as claimed. Further, from the construction of Thurston's automaton (see Definition \ref{ThurA}) it follows that $(R_a \bowtie R_b)^*$ is obtained from the braid $R_b^*$ by adding (from the left) the braid $R_x$, i.e., $(R_a\bowtie R_b)^* = R_b^*R_x$. However, from the construction of Thurston's automaton (see also Lemma \ref{9.1.11}, taking into account $R_x \subseteq \neg R_b^*$), it follows that $(R_a \bowtie R_b)^*$ is a non-repeating braid, then $(R_a \bowtie R_b)^* = R_{b^{-1}x}$ (see Definition \ref{star}) and, using (\ref{Thurstonformulas1}), we get
\[
(R_a\bowtie R_b)^* = R_{b^{-1}x} = bR_x \bigtriangleup bR_b,
\]
but $R_x \supseteq \neg R_b$ i.e., $R_x \cap R_b = \varnothing$. So we obtain
\begin{multline*}
(R_a\bowtie R_b)^* = R_{b^{-1}x} = bR_x \bigtriangleup bR_b = bR_x \cup bR_b = b(aR_a \cap \neg R_b) \cup R_b = \\ =(baR_a \cup R_b) \cap (b\neg R_b  \cup bR_b) = (baR_a \cup R_b) \cap \Omega = baR_a \cup R_b,
\end{multline*}
as claimed. And finally, since we only used the formulas (\ref{Thurstonformulas1}), all these sets satisfy the conditions of Lemma \ref{criteria}.
\end{proof}

\smallskip

\begin{remark}
The condition $aR_a \wedge \neg R_b \ne R_\varepsilon$ is very important for the above mentioned Proposition, otherwise we would have $R_a \asymp R_b = R_a$ and $(R_a \bowtie R_b)^* = R_b^*$, and it is not hard to see that it cannot be true. Moreover, from the equality $(R_a \bowtie R_b)^* = bR_b \cup baR_a$ it does not follow that $R_a \bowtie R_b = R_b \cup aR_a$, because all strands in the ``braid''  $R_b \cup aR_a$ are numbered with respect to some horizontal line which crosses the braid $bR_b \cup baR_a$, it follows that, in general, the set $R_b \cup aR_a$ does not satisfy the conditions of Lemma \ref{criteria}. For example, let us look at the fig.\ref{figCD1} (the left side). We have
\begin{align*}
& aR_a = \{(1,2),(1,3),(1,4),(1,5),(1,6),(2,4),(2,6),(2,5),(4,5)\},\\
& R_b = \{(1,2),(1,4),(1,5),(1,6),(2,4),(2,6),(3,4),(3,5),(3,6),(5,6)\},
\end{align*}
we see that $(4,5), (5,6) \in R_b \cup aR_a$ but $(4,6) \notin R_b \cup aR_a$, i.e., this set does not satisfy to the condition i) of Lemma \ref{criteria}.
\end{remark}

\section{Combinatorial properties of Thurston's operations}
To understand the rewriting procedure, we have to study the disappearing of common maximal (tail) braid. The following proposition will help us to prove some interesting properties of Thurston's operations.

\smallskip

\begin{definition}[the technique of colored strands] Let $R_{a_1}, \ldots, R_{a_\ell} \in D$ be  non-repeating braids, let us consider the braid $R = R_{a_1}\cdots R_{a_\ell}$. Let us color all strands of this braid in different colors $s,s', \ldots, s''$. We will say that the colored strands $s, s',s''$ lay in the braid $R$, and the notation $s,s',\ldots, s'' \in R$ will mean that the braid $R$ contains the corresponding colored strands. We will also say ``the strand $s'$'' instead of ``the strand is colored in the color $s'$''.

\smallskip

\par Let us introduce the following notations
\begin{align*}
&s \asymp s'\mathrm{mod}(R_{a_i}) \mbox{ means that strands $s$ and $s'$ do not cross in the non-repeating braid } R_{a_i},\\
&s \bowtie s' \mathrm{mod}(R_{a_i}) \mbox{ means that strands $s$ and $s'$ cross in the non-repeating braid } R_{a_i},
\end{align*}
here $1 \le i \le \ell$.

\par Since the rewriting procedure preserves the number of crossings then it follows that the phrase ``two strands $s$, $s'$ cross (or do not cross) in a braid which is made by rewriting'' has a well-defined meaning.
\end{definition}

\smallskip

\par The following example may help to understand this concept.

\smallskip

\begin{example}
Let us consider the braid $R = R_aR_bR_c$ which is showed in fig.\ref{dem2},
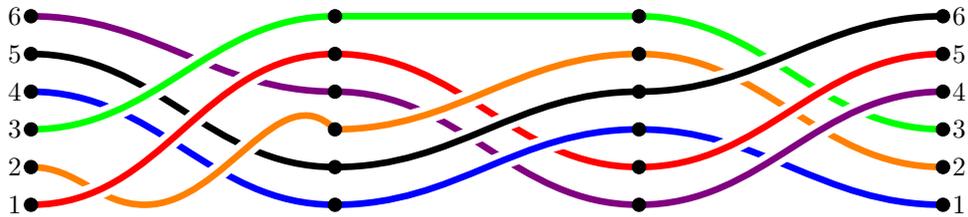
\begin{figure}[h!]
\begin{center}
\begin{tikzpicture}
%~~~~~~~~~~~~~~~~~~~~~~~~~~~~~~~~~~~~~~~~~~~~~~~~~~~~~~~~~~~~~~~~~~~~~~~~~~~~~~~~~~~~~~~~~~~~~~~~
%first left braid
%111111111111111111111111111111111111111111111111111111111111111111111111111111111111111111111111
%sixth(violet)
\draw [line width = 2.5,violet] (0,0) to [out = 0, in = 180] (4,-1);
%fifth(black)
\draw [line width = 2.5] (0,-0.5) to [out = 0, in = 180] (4,-2);
%fourth(blue)
\draw [line width = 7,white] (0,-1) to [out = 0, in = 180] (4,-2.5);
\draw [line width = 2.5,blue] (0,-1) to [out = 0, in = 180] (4,-2.5);
%third(green)
\draw [line width = 7,white] (0,-1.5) to [out = 0, in = 180] (4,0);
\draw [line width = 2.5,green] (0,-1.5) to [out = 0, in = 180] (4,0);
%second(orange)
\draw [line width = 7,white] (0,-2) to [out = 0, in = 180] (1.5,-2.5) to [out = 0 in = 180] (4,-1.5);
\draw [line width = 2.5,orange] (0,-2) to [out = 0, in = 180] (1.5,-2.5) to [out = 0 in = 180] (4,-1.5);
%first(red)
\draw [line width = 7,white] (0,-2.5) to [out = 0, in = 180] (4,-0.5);
\draw [line width = 2.5,red] (0,-2.5) to [out = 0, in = 180] (4,-0.5);
%222222222222222222222222222222222222222222222222222222222222222222222222222222222222222222222222
%second left braid
%sixth(green)
\draw [line width = 2.5,green] (4,0) to [out = 0, in = 180] (8,0);
%fifth(red)
\draw [line width = 2.5,red] (4,-0.5) to [out = 0, in = 180] (8,-2);
%fourth(violet)
\draw [line width = 2.5,violet] (4,-1) to [out = 0, in = 180] (8,-2.5);
%third(orange)
\draw [line width = 7,white] (4,-1.5) to [out = 0, in = 180] (8,-0.5);
\draw [line width = 2.5,orange] (4,-1.5) to [out = 0, in = 180] (8,-0.5);
%second(black)
\draw [line width = 7,white] (4,-2) to [out = 0, in = 180] (8,-1);
\draw [line width = 2.5,black] (4,-2) to [out = 0, in = 180] (8,-1);
%first(blue)
\draw [line width = 7,white] (4,-2.5) to [out = 0, in = 180] (8,-1.5);
\draw [line width = 2.5,blue] (4,-2.5) to [out = 0, in = 180] (8,-1.5);
%333333333333333333333333333333333333333333333333333333333333333333333333333333333333333333333333
%third left braid
%firs(green)
\draw [line width = 2.5,green] (8,0) to [out = 0, in = 180] (12,-1.5);
%fifth(orange)
\draw [line width = 2.5,orange] (8,-0.5) to [out = 0, in = 180] (12,-2);
%fourth(black)
\draw [line width = 7,white] (8,-1) to [out = 0, in = 180] (12,0);
\draw [line width = 2.5,black] (8,-1) to [out = 0, in = 180] (12,0);
%third(blue)
\draw [line width = 2.5,blue] (8,-1.5) to [out = 0, in = 180] (12,-2.5);
%second(red)
\draw [line width = 7,white] (8,-2) to [out = 0, in = 180] (12,-0.5);
\draw [line width = 2.5,red] (8,-2) to [out = 0, in = 180] (12,-0.5);
%first(violet)
\draw [line width = 7,white] (8,-2.5) to [out = 0, in = 180] (12,-1);
\draw [line width = 2.5,violet] (8,-2.5) to [out = 0, in = 180] (12,-1);
%~~~~~~~~~~~~~~~~~~~~~~~~~~~~~~~~~~~~~~~~~~~~~~~~~~~~~~~~~~~~~~~~~~~~~~~~~~~~~~~~~~~~~~~~~~~~~~~~
%................................................................................................
%points
%first left column
\draw [fill](0,0) circle (2.5pt);
\draw [fill](0,-0.5) circle (2.5pt);
\draw [fill](0,-1) circle (2.5pt);
\draw [fill](0,-1.5) circle (2.5pt);
\draw [fill](0,-2) circle (2.5pt);
\draw [fill](0,-2.5) circle (2.5pt);
%second left column
\draw [fill](4,0) circle (2.5pt);
\draw [fill](4,-0.5) circle (2.5pt);
\draw [fill](4,-1) circle (2.5pt);
\draw [fill](4,-1.5) circle (2.5pt);
\draw [fill](4,-2) circle (2.5pt);
\draw [fill](4,-2.5) circle (2.5pt);
%third left column
\draw [fill](8,0) circle (2.5pt);
\draw [fill](8,-0.5) circle (2.5pt);
\draw [fill](8,-1) circle (2.5pt);
\draw [fill](8,-1.5) circle (2.5pt);
\draw [fill](8,-2) circle (2.5pt);
\draw [fill](8,-2.5) circle (2.5pt);
%fourth left column
\draw [fill](12,0) circle (2.5pt);
\draw [fill](12,-0.5) circle (2.5pt);
\draw [fill](12,-1) circle (2.5pt);
\draw [fill](12,-1.5) circle (2.5pt);
\draw [fill](12,-2) circle (2.5pt);
\draw [fill](12,-2.5) circle (2.5pt);
%.................................................................................................
%numbers
\node[left] at (0,0){$6$};
\node[left] at (0,-0.5){$5$};
\node[left] at (0,-1){$4$};
\node[left] at (0,-1.5){$3$};
\node[left] at (0,-2){$2$};
\node[left] at (0,-2.5){$1$};
%----------------------------
\node[right] at (12,0){$6$};
\node[right] at (12,-0.5){$5$};
\node[right] at (12,-1){$4$};
\node[right] at (12,-1.5){$3$};
\node[right] at (12,-2){$2$};
\node[right] at (12,-2.5){$1$};
\end{tikzpicture}
\end{center}
\caption{Here is shown the braid $R$ which is decomposed as $R = R_aR_bR_c$. We start on the left.}\label{dem2}
\end{figure}
here are three permutations
\[
a = \begin{pmatrix}1&2&3&4&5&6\\5&3&6&1&2&4 \end{pmatrix}, \quad b = \begin{pmatrix}1&2&3&4&5&6\\3&4&5&1&2&6 \end{pmatrix}, \quad c = \begin{pmatrix}1&2&3&4&5&6\\4&5&1&6&2&4 \end{pmatrix}
\]
Let us color all its strands in the following way
\begin{gather*}
\mbox{first strand} \mapsto \mbox{red}, \quad \mbox{second strand} \mapsto \mbox{orange}, \quad \mbox{third strand} \mapsto \mbox{green},\\
\mbox{forth strand} \mapsto \mbox{blue}, \quad \mbox{fifth strand} \mapsto \mbox{black}, \quad \mbox{sixth strand} \mapsto \mbox{violet},
\end{gather*}
and let us say that the braid $R$ contains the red, orange, green, blue, black and violet strands. We set
\[
x_{ij}(y) = \begin{cases}1, \mbox{ if the colored strands of i-th and j-th color are crossed in } R_y, \\ 0, \mbox{ otherwise,}  \end{cases}
\]
for any $1 \le i,j \le 6$ and $y \in \{a,b,c\}.$ Then, the crossing of these strands can be described via the following tableau of triples of the form $(x_{ij}(a), x_{ij}(b),x_{ij}(c))$.
\begin{center}
\begin{tabular}{|c|c|c|c|c|c|c|}
\hline
 & red & orange & green & blue & black & violet\\
\hline
red & $(0,0,0)$ & $(1,1,1)$ & $(0,0,1)$ & $(1,1,1)$ & $(1,1,0)$ & $(1,0,0)$\\
\hline
orange & $(1,1,1)$ & $(0,0,0)$ & $(0,0,0)$ & $(1,0,0)$ & $(1,0,1)$ & $(0,1,1)$\\
\hline
green & $(0,0,1)$ & $(0,0,0)$ & $(0,0,0)$ & $(1,0,0)$ & $(1,0,1)$ & $(1,0,1)$\\
\hline
blue & $(1,1,1)$ & $(1,0,1)$ & $(1,0,0)$ & $(0,0,0)$ & $(0,0,0)$ & $(0,1,1)$\\
\hline
black & $(1,1,0)$ & $(1,0,1)$ & $(1,0,1)$ & $(0,0,0)$ & $(0,0,0)$ & $(0,1,0)$\\
\hline
violet & $(1,0,1)$ & $(0,1,1)$ &$(1,0,1)$ &$(0,1,1)$ & $(0,1,0)$ & $(0,0,0)$\\
\hline
\end{tabular}
\end{center}
\end{example}

\smallskip

\par The following lemma describes how the technique of colored strands works.

\smallskip

\begin{lemma}\label{strands}
Let $R_a$ and $R_b$ be non-repeating braids, and let us assume that $aR_a \wedge \neg R_b \ne R_\varepsilon$. Let $s$ and $s'$ be two strands of the braid $R_aR_b$, then we have the following corollaries;
\begin{align*}
&s \asymp s' \mathrm{mod}(a \asymp b) \Longleftrightarrow s \asymp s' \mathrm{mod}(a) \mbox{ or } \begin{cases} s \bowtie s' \mathrm{mod}(a)\\ s \asymp s' \mathrm{mod}(b) \end{cases}\\
&s \bowtie s' \mathrm{mod}(a \asymp b) \Longleftrightarrow \begin{cases}s \bowtie s' \mathrm{mod}(a)\\s \bowtie s' \mathrm{mod}(b) \end{cases}\\
&s \asymp s' \mathrm{mod}(a \bowtie b) \Longleftrightarrow \begin{cases} s \asymp s' \mathrm{mod}(a)\\ s \asymp s' \mathrm{mod}(b) \end{cases}\\
&s \bowtie s' \mathrm{mod}(a \bowtie b) \Longleftrightarrow s \bowtie s'\mathrm{mod}(b)  \mbox{ or } \begin{cases} s \bowtie s' \mathrm{mod}(a)\\ s \asymp s' \mathrm{mod}(b). \end{cases}
\end{align*}
\end{lemma}
\begin{proof} The proof is immediately follows from the construction of Thurston's automaton (see Definition \ref{ThurA}).
\end{proof}

\smallskip

\begin{proposition}\label{stop} %???? ?? ????????? ?????????, ?? ?????? ???????????? ????? ??????????%
For any three non-repeating braids $R_a$, $R_b$ and $R_c$ we have
\begin{align}
&R_a^* \wedge \neg R_b = R_\varepsilon \Longrightarrow (R_a \bowtie (R_b \asymp R_c))^* \wedge \neg (R_b \bowtie R_c) = R_\varepsilon,\label{stop1}\\
&R_b^* \wedge \neg R_c = R_\varepsilon \Longrightarrow (R_a \asymp R_b)^* \wedge \neg ((R_a \bowtie R_b)\asymp R_c) = R_\varepsilon,\label{stop2}\\
&(R_a \asymp (R_b \asymp R_c))^*\wedge \neg \left((R_a \bowtie (R_b \asymp R_c)) \asymp (R_b \bowtie R_c) \right) = R_\varepsilon,\label{stop3}\\
&((R_{a} \asymp R_b) \bowtie ((R_a \bowtie R_b) \asymp R_c))^* \wedge \neg  ((R_a \bowtie R_b) \bowtie R_c) = R_\varepsilon.\label{stop4}
\end{align}
\end{proposition}
\begin{proof} Let us consider some strands $s$ and $s'$ of the braid $R_aR_bR_c$. To prove this proposition, we will use the machinery of colored strands and Lemma \ref{strands}.

\smallskip

\par i) Let $R^* \wedge \neg R_b = R_\varepsilon$. Let us consider two strands $s,s'$ such that $(s,s') \in (R_a \bowtie (R_b \asymp R_c))^* \cap \neg (R_b \bowtie R_c)$, i.e., $s \bowtie s' \bmod(a \bowtie (b \asymp c))$ and $s \asymp s' \bmod(b \bowtie c)$, we have
\begin{multline*}
s \bowtie s' \bmod(a \bowtie (b \asymp c)) \Longleftrightarrow s \bowtie s' \bmod (b \asymp c) \mbox{ or } \begin{cases} s \bowtie s' \bmod (a)\\s \asymp s' \bmod(b \asymp c)\end{cases} \Longleftrightarrow \\ \Longleftrightarrow \begin{cases} s \bowtie s'\bmod (b)\\ s \bowtie s' \bmod(c) \end{cases} \mbox{ or } \begin{cases} s \bowtie s' \bmod(a)\\ s \asymp s' \bmod(b) \mbox{ or } \begin{cases}s \bowtie s' \bmod(b)\\ s \asymp s' \bmod(c). \end{cases} \end{cases}
\end{multline*}
on the other hand, we have
\[
s \asymp s' \bmod(b \bowtie c) \Longleftrightarrow \begin{cases}s \asymp s' \bmod(b)\\s \asymp s' \bmod(c) \end{cases}
\]
if we assume that $R_a^* \cap \neg R_b = \varnothing$ then $(R_a \bowtie (R_b \asymp R_c))^* \cap \neg (R_b \bowtie R_c) = \varnothing$. Otherwise, we have to put  $\begin{cases} s \bowtie s' \bmod(a) \\ s \asymp s' \bmod(b)\\s \asymp s' \bmod(c)\end{cases}$ but this means that
\[
(R_a \bowtie (R_b \asymp R_c))^* \wedge \neg (R_b \bowtie R_c) = R_a^* \wedge \neg R_b = R_\varepsilon,
\]
as claimed.

\smallskip

\par ii) Let $R_b^* \wedge \neg R_c = R_\varepsilon$. Let us consider two strands $(s,s')\in (R_a \asymp R_b)^* \cap \neg ((R_a \bowtie R_b)\asymp R_c)$, we have
\[
s \bowtie s' \bmod(a \asymp b) \Longleftrightarrow \begin{cases}s \bowtie s' \bmod(a)\\ s \bowtie s'\bmod(b) \end{cases}
\]
on the other hand,
\begin{multline*}
s \asymp s' \bmod((a \bowtie b) \asymp c) \Longleftrightarrow s \asymp s' \bmod(a \bowtie b) \mbox{ or } \begin{cases} s \bowtie s' \bmod(a \bowtie b)\\ s \asymp s' \bmod(c) \end{cases} \Longleftrightarrow \\ \Longleftrightarrow \begin{cases} s \asymp s' \bmod(a) \\ s \asymp s' \bmod(b) \end{cases} \mbox{ or } \begin{cases} s \bowtie s' \bmod(b) \mbox{ or } \begin{cases}s \bowtie s' \bmod(a)\\s \asymp s' \bmod(b) \end{cases}\\s\asymp s' \bmod(c) \end{cases}
\end{multline*}
it is not hard to see that if we assume that $R_b^* \cap \neg R_c = \varnothing$ then $(R_a \asymp R_b)^* \cap \neg ((R_a \bowtie R_b)\asymp R_c) = \varnothing$. Otherwise, we have to put  $\begin{cases}s\bowtie s' \bmod(b)\\s \asymp s' \bmod(c) \end{cases}$  it follows that
\[
(R_a \asymp R_b)^* \wedge \neg ((R_a \bowtie R_b)\asymp R_c) = R_b^* \wedge \neg R_c = R_\varepsilon,
\]
as claimed.

\smallskip

\par iii) Let us assume that $s \bowtie s' \mathrm{mod}(a \asymp (b \asymp c))$ and $s \asymp s' \mathrm{mod}((a \bowtie (b \asymp c))\asymp (b \bowtie c))$, we have
$$
s \bowtie s' \mathrm{mod}(a \asymp (b \asymp c)) \Longleftrightarrow \begin{cases} s \bowtie s' \mathrm{mod}(a)\\s \bowtie s' \mathrm{mod}(b)\\s \bowtie s' \mathrm{mod}(c) \end{cases}
$$
and
\begin{multline*}
s \asymp s' \mathrm{mod}((a \bowtie (b \asymp c))\asymp (b \bowtie c)) \Longleftrightarrow s \asymp s' \mathrm{mod}(a \bowtie (b \asymp c)) \mbox{ or } \begin{cases} s \bowtie s' \mathrm{mod}(a \bowtie (b \asymp c)) \\ s \asymp s' \mathrm{mod}(b \bowtie c) \end{cases} \Longleftrightarrow \\ \Longleftrightarrow \begin{cases}s \asymp s' \mathrm{mod}(a) \\ s \asymp s' \mathrm{mod}(b \asymp c)   \end{cases} \mbox{ or } \begin{cases} s \bowtie s' \mathrm{mod}(b \asymp c) \mbox{ or } \begin{cases}s \bowtie s' \mathrm{mod}(a)\\ s \asymp s' \mathrm{mod}(b \asymp c)  \end{cases}\\ \begin{cases}s \asymp s' \mathrm{mod}(b) \\ s \asymp s' \mathrm{mod}(c)  \end{cases}  \end{cases} \Longleftrightarrow \\ \Longleftrightarrow \begin{cases} s \asymp s' \mathrm{mod}(a)\\ s \asymp s' \mathrm{mod}(b) \mbox{ or }  \begin{cases} s \bowtie s' \mathrm{mod}(b)\\ s \asymp s' \mathrm{mod}(c) \end{cases}  \end{cases} \mbox{ or } \begin{cases} \begin{cases} s \bowtie s' \mathrm{mod}(b) \\ s \bowtie s'\mathrm{mod}(c)  \end{cases} \mbox{ or } \begin{cases}s \bowtie s'\mathrm{mod}(a) \\ s \asymp s' \mathrm{mod}(b) \mbox{ or } \begin{cases} s \bowtie s' \mathrm{mod}(b)\\ s \asymp s' \mathrm{mod}(c) \end{cases} \end{cases} \\ \begin{cases} s \asymp s' \mathrm{mod}(b) \\ s \asymp s' \mathrm{mod}(c).  \end{cases} \end{cases}
\end{multline*}
\par It is not hard to see that we get the following corollary
$$
\begin{cases} \begin{cases} s \bowtie s' \mathrm{mod}(b) \\ s \bowtie s'\mathrm{mod}(c)  \end{cases} \mbox{ or } \begin{cases}s \bowtie s'\mathrm{mod}(a) \\ s \asymp s' \mathrm{mod}(b) \mbox{ or } \begin{cases} s \bowtie s' \mathrm{mod}(b)\\ s \asymp s' \mathrm{mod}(c) \end{cases}   \end{cases} \\ \begin{cases} s \asymp s' \mathrm{mod}(b), \\ s \asymp s' \mathrm{mod}(c)  \end{cases} \end{cases} \Longleftrightarrow s \asymp s' \mathrm{mod}(c),
$$
this means that $(a \asymp (b \asymp c))R_{a \asymp (b \asymp c)} \cap \neg R_{(a \bowtie (b \asymp c)) \asymp (b \bowtie c)} = \varnothing$, as claimed.

\smallskip

\par iv) Let us consider now the braid $(R_{a} \asymp R_b) \bowtie ((R_a \bowtie R_b) \asymp R_c)$, and let us assume that for some strands $s,$ $s'$ of the braid $R_aR_bR_c$ we have $s \bowtie s' \mathrm{mod}((a \asymp b)\bowtie ((a \bowtie b)\asymp c))$, we have
\begin{multline*}
s \bowtie s' \mathrm{mod}((a \asymp b)\bowtie ((a \bowtie b)\asymp c)) \Longleftrightarrow
s \bowtie s' \mathrm{mod}((a \bowtie b)\asymp c)) \mbox{ or } \begin{cases} s \bowtie s' \mathrm{mod}(a \asymp b)\\ s \asymp s' \mathrm{mod}((a \bowtie b)\asymp c)) \end{cases} \Longleftrightarrow \\ \Longleftrightarrow  \begin{cases} s \bowtie s' \mathrm{mod}(a \bowtie b)\\ s \bowtie s' \mathrm{mod}(c) \end{cases}  \mbox{ or } \begin{cases} \begin{cases} s \bowtie s' \mathrm{mod}(a)\\ s \bowtie s' \mathrm{mod}(b) \end{cases}\\ \left(s \asymp s' \mathrm{mod}(a \bowtie b)\right) \mbox{ or }\begin{cases} s \bowtie s' \mathrm{mod}(a \bowtie b)\\s \asymp s' \mathrm{mod}(c) \end{cases} \end{cases} \Longleftrightarrow \\ \Longleftrightarrow \begin{cases} s \bowtie s' \mathrm{mod}(b) \mbox{ or } \begin{cases} s \bowtie s' \mathrm{mod}(a)\\ s \asymp s'\mathrm{mod}(b)  \end{cases}\\ s \bowtie s'\mathrm{mod}(c) \end{cases} \mbox{ or } \begin{cases} \begin{cases}s \bowtie s' \mathrm{mod}(a)\\ s \bowtie s' \mathrm{mod}(b) \end{cases}\\ \begin{cases}s \asymp s' \mathrm{mod}(a)\\ s \asymp s' \mathrm{mod}(b)  \end{cases} \mbox{ or } \begin{cases} s \bowtie s' \mathrm{mod}(b) \mbox{ or } \begin{cases} s \bowtie s' \mathrm{mod}(a)\\ s \asymp s'\mathrm{mod}(b) \end{cases} \\ s \asymp s'\mathrm{mod}(c) \end{cases} \end{cases}
\end{multline*}
it is not hard to see that
$$
\begin{cases} \begin{cases}s \bowtie s' \mathrm{mod}(a)\\ s \bowtie s' \mathrm{mod}(b) \end{cases}\\ \begin{cases}s \asymp s' \mathrm{mod}(a)\\ s \asymp s' \mathrm{mod}(b)  \end{cases} \mbox{ or } \begin{cases} s \bowtie s' \mathrm{mod}(b) \mbox{ or } \begin{cases} s \bowtie s' \mathrm{mod}(a)\\ s \asymp s'\mathrm{mod}(b) \end{cases} \\ s \asymp s'\mathrm{mod}(c) \end{cases} \end{cases} \Longleftrightarrow \begin{cases} s \bowtie s' \mathrm{mod}(a)\\ s \bowtie s'\mathrm{mod}(b)\\ s \asymp s' \mathrm{mod}(c)  \end{cases}
$$
i.e., we have
\begin{multline*}
s \bowtie s' \mathrm{mod}((a \asymp b) \bowtie ((a \bowtie b) \asymp c)) \Longleftrightarrow \\ \Longleftrightarrow  \begin{cases} s \bowtie s' \mathrm{mod}(b) \mbox{ or } \begin{cases} s \bowtie s' \mathrm{mod}(a)\\ s \asymp s'\mathrm{mod}(b)  \end{cases}\\ s \bowtie s'\mathrm{mod}(c) \end{cases} \mbox{ or } \begin{cases} s \bowtie s' \mathrm{mod}(a)\\ s \bowtie s'\mathrm{mod}(b)\\ s \asymp s' \mathrm{mod}(c).  \end{cases}
\end{multline*}

\par On the other hand, we have
$$
s \asymp s' \mathrm{mod}((a \bowtie b) \bowtie c) \Longleftrightarrow \begin{cases} s \asymp s' \mathrm{mod}(a)\\s \asymp s' \mathrm{mod}(b) \\ s \asymp s' \mathrm{mod}(c), \end{cases}
$$
it follows that $((a \asymp b) \bowtie ((a \bowtie b) \asymp c))R_{(a \asymp b) \bowtie ((a \bowtie b) \asymp c)} \cap \neg R_{(a \bowtie b) \bowtie c} = \varnothing,$ as claimed.
\end{proof}

\smallskip

\begin{theorem}\label{GSB} The triple $(D,\asymp,\bowtie)$ with binary operations $\asymp$ and $\bowtie$ satisfies the following equations for any non-repeating braids $R_a,R_b,R_c \in D$:
\begin{align}
&R_a \asymp R_a = R_a, \quad R_a \bowtie R_a = R_a\\
& R_a \asymp R_b = R_a \Longleftrightarrow R_a \bowtie R_b = R_b\\
&R_a \asymp (R_a \bowtie R_b) = R_a = R_a \bowtie (R_a \bowtie R_b),\quad (R_a \asymp R_b)\asymp R_b = R_b = (R_a \asymp R_b)\bowtie R_b\\
&R_a \asymp (R_b \asymp R_c) = (R_a \asymp R_b) \asymp ((R_a \bowtie R_b) \asymp R_c),\label{l}\\
&(R_a \bowtie (R_b \asymp R_c)) \asymp (R_b \bowtie R_c) = (R_a \asymp R_b) \bowtie ((R_a \bowtie R_b) \asymp R_c),\label{m}\\
&(R_a \bowtie (R_b \asymp R_c)) \bowtie (R_b \bowtie R_c) = (R_a \bowtie R_b) \bowtie R_c,\label{r}\\
& R_a \asymp R_b = R_b \ne R_a, \Longleftrightarrow R_a = R_{xb}, \mbox{ where } xb = bx,\,b^2 = 1 \mbox{ in the permutation group.}
\end{align}
\end{theorem}
\begin{proof}
\par i) The formulas $R_a \asymp R_a = R_a,$ $R_a \bowtie R_a = R_a,$ and $R_a \asymp R_b = R_a \Longleftrightarrow R_a \bowtie R_b = R_b,$ immediately follow from the definition of the operations $\asymp$ and $\bowtie$.

\smallskip

\par ii) Let us remark that from the construction of Thurston's automaton it follows that  $R_a^* \cap \neg (R_a \bowtie R_b) = \varnothing$ and $(R_a \asymp R_b)^* \cap \neg R_b = \varnothing$. Since we have taken the common maximal braid, it follows that $R_a (R_a\bowtie R_b)$ and $(R_a \asymp R_b)R_b$ are the greedy normal forms, i.e,
\[
R_a \asymp (R_a \bowtie R_b) = R_a = R_a \bowtie (R_a \bowtie R_b),\quad (R_a \asymp R_b)\asymp R_b = R_b = (R_a \asymp R_b)\bowtie R_b.
\]

\smallskip

\par iii) Since this machinery works iff there exists common maximal braid, we have to consider the combinations of the following possibilities; 1) $a \asymp b = a$, $a \bowtie b = b$, 2) $b \asymp c = b$, $b \bowtie c = c$, 3) $(a \bowtie b) \asymp c = a \bowtie b$, $(a \bowtie b) \bowtie c = c$, 4) $a \asymp (b \asymp c) = a$, $a \bowtie (b \asymp c) = b \asymp c)$, 5) all above mentioned equations are false.

\par iii--1) If we assume one of the following possibilities
\[
\begin{cases} a \asymp b = a,\, a \bowtie b = b, \\ a \asymp (b \asymp c) = a,\, a \bowtie(b \asymp c) = b \asymp c \end{cases} \quad \begin{cases} b \asymp c = c, \, b \bowtie c = c\\ (a \bowtie b) \asymp c = a\bowtie b,\, (a \bowtie b)\bowtie c = c \end{cases}\quad \begin{cases}a \asymp b = a,\, a \bowtie b = b\\b \asymp c = b,\, b \bowtie c = c, \end{cases}
\]
we get the trivial equations,
\begin{align*}
& R_a = R_a,\\
& R_b \asymp R_c = R_b \asymp R_c,\\
& R_b \bowtie R_c = R_b \bowtie R_c.
\end{align*}

\smallskip

\par iii--2) Let us consider one of the following possibilities $\begin{cases} a \asymp b = a\\a \bowtie b = b\end{cases}$ or $\begin{cases} b \asymp c = b\\b\bowtie c = c\end{cases}$ then we have to prove the following formulas (respectively);
\begin{align*}
& R_a \asymp (R_b \asymp R_c) = R_a \asymp (R_b \asymp R_c),\\
&(R_a \bowtie (R_b \asymp R_c)) \asymp (R_b \bowtie R_c) = R_a \bowtie (R_b \asymp R_c),\\
&(R_a \bowtie (R_b \asymp R_c)) \bowtie (R_b \bowtie R_c) = R_b \bowtie R_c.
\end{align*}
we see that the first formula is a trivial equation, another two formulas immediately follow from Proposition \ref{stop} (see (\ref{stop1})). Further, if we put $\begin{cases} b \asymp c = b\\b\bowtie c = c\end{cases}$ then we have to prove that
\begin{align*}
&R_a \asymp R_b = (R_a \asymp R_b) \asymp ((R_a \bowtie R_b) \asymp R_c),\\
&(R_a \bowtie R_b) \asymp R_c = (R_a \asymp R_b) \bowtie ((R_a \bowtie R_b) \asymp R_c),\\
&(R_a \bowtie R_b) \bowtie R_c = (R_a \bowtie R_b) \bowtie R_c,
\end{align*}
we see that the last formula is a trivial equation and another two also follow from Proposition \ref{stop} (see (\ref{stop2})).

\smallskip

\par iii--4) Now, let us assume that all necessary common maximal braids exist. First of all, let us remark that it suffices to prove (\ref{l}) and (\ref{r}), because all these braids are made from the braid $R_aR_bR_c$ by moving the crossings of some strands. Then, if we get the same crossing in the head of a braid and in its tail, then the middle braids will be the same. Let us prove (\ref{l}). We have
$$
s \bowtie s' \mathrm{mod}(a \asymp (b \asymp c)) \Longleftrightarrow \begin{cases} s \bowtie s' \mathrm{mod}(a)\\ s \bowtie s' \mathrm{mod}(b)\\ s \bowtie s' \mathrm{mod}(c), \end{cases}
$$
and, on the other hand, we have
\begin{multline*}
s \bowtie s' \mathrm{mod}((a \asymp b) \asymp ((a \bowtie b) \asymp c)) \Longleftrightarrow \begin{cases} s \bowtie s' \mathrm{mod}(a \asymp b) \\ s \bowtie s' \mathrm{mod}((a \bowtie b) \asymp c)  \end{cases} \Longleftrightarrow \begin{cases} s \bowtie s' \mathrm{mod}(a)\\ s \bowtie s' \mathrm{mod}(b) \\ s \bowtie s' \mathrm{mod}(a \bowtie b) \\ s \bowtie s' \mathrm{mod}(c)  \end{cases} \Longleftrightarrow \\ \Longleftrightarrow \begin{cases} s \bowtie s' \mathrm{mod}(a)\\ s \bowtie s' \mathrm{mod}(b) \\ s \bowtie s' \mathrm{mod}(b) \mbox{ or } \begin{cases} s \bowtie s' \mathrm{mod}(a) \\ s \asymp s' \mathrm{mod}(b) \end{cases} \\ s \bowtie s' \mathrm{mod}(c)  \end{cases} \Longleftrightarrow \begin{cases} s \bowtie s' \mathrm{mod}(a)\\ s \bowtie s' \mathrm{mod}(b)\\ s \bowtie s' \mathrm{mod}(c), \end{cases}
\end{multline*}
i.e., we get $R_a \asymp (R_b \asymp R_c) = (R_a \asymp R_b) \asymp ((R_a \bowtie R_b) \asymp R_c)$.

\smallskip

\par iii--5) Let $s \bowtie s' \mathrm{mod}((a \bowtie(b \asymp c))\bowtie (b \bowtie c))$ then we have
\begin{multline*}
s \bowtie s' \mathrm{mod}((a \bowtie(b \asymp c))\bowtie (b \bowtie c)) \Longleftrightarrow s \bowtie s' \mathrm{mod}(b \bowtie c) \mbox{ or } \begin{cases} s \bowtie s'\mathrm{mod}(a \bowtie (b \asymp c))\\ s \asymp s' \mathrm{mod}(b \bowtie c) \end{cases} \Longleftrightarrow \\ \Longleftrightarrow  s \bowtie s' \mathrm{mod}(c) \mbox{ or } \begin{cases} s \bowtie s' \mathrm{mod}(b)\\ s \asymp s' \mathrm{mod}(c) \end{cases} \mbox{ or } \begin{cases}s \bowtie s' \mathrm{mod}(b \asymp c) \mbox{ or } \begin{cases} s \bowtie s' \mathrm{mod}(a)\\ s \asymp s' \mathrm{mod}(b \asymp c) \end{cases}\\ s \asymp s' \mathrm{mod}(b), \\ s \asymp s' \mathrm{mod}(c)  \end{cases} \Longleftrightarrow \\ \Longleftrightarrow s \bowtie s' \mathrm{mod}(c) \mbox{ or } \begin{cases} s \bowtie s' \mathrm{mod}(b)\\ s \asymp s' \mathrm{mod}(c) \end{cases} \mbox{ or } \begin{cases} \begin{cases} s \bowtie s' \mathrm{mod}(b) \\ s \bowtie s' \mathrm{mod}(c)\end{cases} \mbox{ or } \begin{cases} s \bowtie s' \mathrm{mod}(a) \\ s \asymp s' \mathrm{mod}(b) \mbox{ or } \begin{cases} s \bowtie s' \mathrm{mod}(b) \\ s \asymp s' \mathrm{mod}(c)   \end{cases}  \end{cases}\\ s \asymp s' \mathrm{mod}(b)\\ s \asymp s' \mathrm{mod}(c)  \end{cases}
\end{multline*}
we see that
$$
\begin{cases} \begin{cases} s \bowtie s' \mathrm{mod}(b), \\ s \bowtie s' \mathrm{mod}(c)\end{cases} \mbox{ or } \begin{cases} s \bowtie s' \mathrm{mod}(a), \\ s \asymp s' \mathrm{mod}(b) \mbox{ or } \begin{cases} s \bowtie s' \mathrm{mod}(b), \\ s \asymp s' \mathrm{mod}(c)   \end{cases}  \end{cases}\\ s \asymp s' \mathrm{mod}(b)\\ s \asymp s' \mathrm{mod}(c)  \end{cases} \Longleftrightarrow \begin{cases} s \bowtie s' \mathrm{mod}(a)\\ s \asymp s' \mathrm{mod}(b)\\s \asymp s' \mathrm{mod}(c)  \end{cases}
$$
i.e., we arrive at
$$
s \bowtie s' \mathrm{mod}((a \bowtie(b \asymp c))\bowtie (b \bowtie c)) \Longleftrightarrow s \bowtie s' \mathrm{mod}(c) \mbox{ or } \begin{cases} s \bowtie s' \mathrm{mod}(b)\\ s \asymp s' \mathrm{mod}(c) \end{cases} \mbox{ or } \begin{cases} s \bowtie s' \mathrm{mod}(a)\\ s \asymp s' \mathrm{mod}(b)\\s \asymp s' \mathrm{mod}(c). \end{cases}
$$
\par On the other hand, let $s \bowtie s' \mathrm{mod}((a \bowtie b) \bowtie c)$. We have
\begin{multline*}
s \bowtie s' \mathrm{mod}((a \bowtie b) \bowtie c) \Longleftrightarrow  s \bowtie s'\mathrm{mod}(c) \mbox{ or } \begin{cases} s \bowtie s'\mathrm{mod}(a \bowtie b)\\ s \asymp s'\mathrm{mod}(c)  \end{cases} \Longleftrightarrow \\ \Longleftrightarrow s \bowtie s'\mathrm{mod}(c) \mbox{ or } \begin{cases}s \bowtie s'\mathrm{mod}(b) \mbox{ or }  \begin{cases} s \bowtie s'\mathrm{mod}(a)\\ s \asymp s'\mathrm{mod}(b) \end{cases} \\ s \asymp s'\mathrm{mod}(c) \end{cases} \Longleftrightarrow \\ \Longleftrightarrow  s \bowtie s' \mathrm{mod}(c) \mbox{ or } \begin{cases} s \bowtie s' \mathrm{mod}(b)\\ s \asymp s' \mathrm{mod}(c) \end{cases} \mbox{ or } \begin{cases} s \bowtie s' \mathrm{mod}(a)\\ s \asymp s' \mathrm{mod}(b)\\s \asymp s' \mathrm{mod}(c), \end{cases}
\end{multline*}
i.e., $(R_a \bowtie (R_b \asymp R_c)) \bowtie (R_b \bowtie R_c) = (R_a \bowtie R_b) \bowtie R_c$.

\smallskip

\par iv) Let us prove that if $R_a \asymp R_b = R_b \ne R_a$, then the braid $R_a$ can be expressed as $R_a = R_{xb}$, where $x$ and $b$ are commutative permutations $xb = bx$, and for the permutation $b$ we have $b^2 = 1$.

\par Let us assume that we have a non-repeating braid $R_b$. We have to add a new crossing (in the upper side of this braid). We see that the equation $R_a \asymp R_b = R_b$ is possible iff $\neg R_b \nsupseteq bR_b$, but it is equivalent to the condition $R_b \supseteq bR_b$ and, since the number of elements of these sets are equal, then we get $R_b = bR_b$. It follows that $R_b = R_{b^{-1}}$, i.e., $b^2= 1$, as claimed. It is clear that the adding of a new crossing to the braid $R_b$ is equivalent to multiplication $R_xR_b$, i.e., we have $R_a = R_xR_b$. It is not hard to see that $xR_x \subseteq \neg R_b$, because $R_x$ means new crossing of the non-crossing strands of $R_b$. But from Lemma \ref{criteria} it follows that $R_xR_b =R_{bx}.$ Since $R_aR_b = R_bR_a$, we have $R_aR_b = R_xR_bR_b$ and, on the other hand, $R_bR_a = R_bR_xR_b$, i.e, $R_xR_bR_b = R_bR_xR_b$. It follows that $R_xR_b = R_bR_x$, but since $R_xR_b \in D$ (i.e., it is a non-repeating braid), then from Lemma \ref{criteria} it follows that $bR_b \subseteq \neg R_x$, i.e., we get $R_{xb} = R_{bx}$, which means that $xb = bx$.

\smallskip

\par The proof is completed.
\end{proof}

\smallskip

\begin{definition}[An order on the Thurston's generators]\label{order}
For any non-repeating braid $R_\pi$ we can define \cite[Lemma 9.1.5]{EpThur} its length $|R_\pi|$ as a power of the set $R_\pi$. Let us write any non-repeating braid $R_\pi \in D$ as follows:
\[
R_\pi = \left\{\underbrace{(1,i_{11}), \ldots (1,i_{1\ell_1})}_{\mbox{{\tiny all crossings of the first strand}}}, \ldots, \underbrace{(j,i_{j1}), \ldots, (j,i_{j\ell_j})}_{\mbox{{\tiny all crossings of the j-th strand}}}, \ldots, (n-1, i_{n-1,1})\right\},
\]
where $i_{11} < \ldots < i_{i\ell_1}$, $\ldots$, $i_{j1} < \ldots < i_{j\ell_j}$ etc. Let us set $1<2<\ldots<n$, then we can order $D$ deg-lexicographically. We denote this order by $\preceq$.
\end{definition}

\smallskip

\begin{theorem} For the braid monoid $\mathbf{Br}_n^+$, and for the set of non-repeating braids $D$ let us consider the order $\preceq$ as above. Then the following set of relations
\[
\mathcal{R} = \{R_aR_b = (R_a \asymp R_b)(R_a \bowtie R_b)\}
\]
is a Gr\"obner --- Shirshov basis for the braid monoid in Thurston's generators (non-repeating braids).
\end{theorem}
\begin{proof}
Indeed, let us consider the Buchberger --- Shirshov's algorithm for the word $R_aR_bR_c$. We get
\begin{multline*}
[R_a|R_b|R_c] \to [R_a \asymp R_b| R_a \bowtie R_b|R_c] \to [R_a \asymp R_b|(R_a \bowtie R_b) \asymp R_c|(R_a \bowtie R_b) \bowtie R_c] \to\\ \to  [(R_a \asymp R_b) \asymp ((R_a \bowtie R_b)\asymp R_c)|(R_a \asymp R_b) \bowtie ((R_a \bowtie R_b)\asymp R_c)|R_b \bowtie R_c],
\end{multline*}
on the other hand,
\begin{multline*}
[R_a|R_b|R_c] \to [R_a| R_b \asymp R_c|R_b \bowtie R_c] \to [R_a \asymp (R_b \asymp R_c)|R_a \bowtie (R_b \asymp R_c)|R_b \bowtie R_c] \to \\ \to [R_a \asymp (R_b \asymp R_c)|(R_a \bowtie (R_b \asymp R_c))\asymp (R_b \bowtie R_c)|(R_a \bowtie (R_b \asymp R_c)) \bowtie (R_b \bowtie R_c)],
\end{multline*}
then from Lemma \ref{stop} and Theorem \ref{GSB} it follows that $\mathcal{R}$ is a Gr\"obner --- Shirshov basis, as claimed.
\end{proof}

\smallskip

\paragraph{Garside's braid and flip's involution.} There is an important element $\Omega = \Omega_n$ (Garside's braid or Garside's element), described physically as the $180^\circ$ clockwise rotation of the $n$ strands together. This braid corresponds to the permutation $\omega = \begin{pmatrix}1&2& \ldots &n-1& n \\ n& n-1 & \ldots &2 &1 \end{pmatrix}$ which is the maximal element of $\mathbb{S}_n$ with respect to the above mentioned order.

\smallskip

\par To find out more about $\Omega$, we look at a semigroup automorphism of $A^*$ called a {\it flip}, which takes each generator $\sigma_i$ to $\widetilde{\sigma_i}: = \sigma_{n-i}$. The name is justified, because the image of a braid under this automorphism is indeed obtained by flipping this braid around the horizontal axis.

\par Of course, the relation $\Omega R = \widetilde{R}\Omega$ is contained in the Gr\"obner --- Shirshov basis, because the set $\Omega^* \wedge \neg R = \Omega \wedge \neg R = \neg R$ is always a braid. Then the rewriting procedure looks like completing the braid $R$ to the braid $\Omega$, meanwhile $\Omega$ is transformed to $\widetilde{R}$ (see fig. \ref{RO->}).

\begin{figure}[h!]
\begin{center}
\begin{tikzpicture}
%5
\draw [line width = 2.5,orange] (2,0) to [out = 270,in=90] (0,-4);
%4
\draw [line width = 6,white] (1.5,0) to [out = 270,in=90] (0,-3) to [out = 270, in = 90] (0.5,-4);
\draw [line width = 2.5,violet] (1.5,0) to [out = 270,in=90] (0,-3) to [out = 270, in = 90] (0.5,-4);
%3
\draw [line width = 6,white] (1,0) to [out = 270,in=90] (0,-2) to [out = 270, in = 90] (1,-4);
\draw [line width = 2.5,green] (1,0) to [out = 270,in=90] (0,-2) to [out = 270, in = 90] (1,-4);
%2
\draw [line width = 6,white] (0.5,0) to [out = 270,in=90] (0,-1) to [out = 270, in = 90] (1.5,-4);
\draw [line width = 2.5,blue] (0.5,0) to [out = 270,in=90] (0,-1) to [out = 270, in = 90] (1.5,-4);
%1
\draw [line width = 7,white] (0,0) to [out = 270,in=90] (2,-4);
\draw [line width = 2.5,red] (0,0) to [out = 270,in=90] (2,-4);
%-------------------------------------------------------------
\draw [line width = 2.5,red] (2,-4) to [out = 270,in=90] (2,-6);
\draw [line width = 2.5,cyan] (1.5,-4) to [out = 270,in=90] (1,-6);
\draw [line width = 7,white] (1,-4) to [out = 270,in=90] (1.5,-6);
\draw [line width = 2.5,olive] (1,-4) to [out = 270,in=90] (1.5,-6);
\draw [line width = 2.5,gray] (0.5,-4) to [out = 270,in=90] (0,-6);
\draw [line width = 7,white] (0,-4) to [out = 270,in=90] (0.5,-6);
\draw [line width = 2.5,black] (0,-4) to [out = 270,in=90] (0.5,-6);
%--------------------------------
%RIGHT
%5
\draw [line width = 2.5,gray] (7,-2) to [out = 270,in=90] (5,-6);
%4
\draw [line width = 6,white] (6.5,-2) to [out = 270,in=90] (5,-5) to [out = 270, in = 90] (5.5,-6);
\draw [line width = 2.5,black] (6.5,-2) to [out = 270,in=90] (5,-5) to [out = 270, in = 90] (5.5,-6);
%3
\draw [line width = 6,white] (6,-2) to [out = 270,in=90] (5,-4) to [out = 270, in = 90] (6,-6);
\draw [line width = 2.5,cyan] (6,-2) to [out = 270,in=90] (5,-4) to [out = 270, in = 90] (6,-6);
%2
\draw [line width = 6,white] (5.5,-3) to [out = 270,in=90] (5,-3) to [out = 270, in = 90] (6.5,-6);
\draw [line width = 2.5,olive] (5.5,-2) to [out = 270,in=90] (5,-3) to [out = 270, in = 90] (6.5,-6);
%1
\draw [line width = 7,white] (5,-2) to [out = 270,in=90] (7,-6);
\draw [line width = 2.5,red] (5,-2) to [out = 270,in=90] (7,-6);
%-------------------------------------------------------------------
\draw [line width = 2.5,orange] (7,0) to [out = 270,in=90] (6.5,-2);
\draw [line width = 7,white] (6.5,0) to [out = 270,in=90] (7,-2);
\draw [line width = 2.5,violet] (6.5,0) to [out = 270,in=90] (7,-2);
\draw [line width = 2.5,green] (6,0) to [out = 270,in=90] (5.5,-2);
\draw [line width = 7,white] (5.5,0) to [out = 270,in=90] (6,-2);
\draw [line width = 2.5,blue] (5.5,0) to [out = 270,in=90] (6,-2);
\draw [line width = 2.5,red] (5,0) to [out = 270,in=90] (5,-2);
%--------------------------------------------------------------------
\draw [fill](0,0) circle (2.5pt);
\draw [fill](0.5,0) circle (2.5pt);
\draw [fill](1,0) circle (2.5pt);
\draw [fill](1.5,0) circle (2.5pt);
\draw [fill](2,0) circle (2.5pt);
%--------------------------------
\draw [fill](0,-4) circle (2.5pt);
\draw [fill](0.5,-4) circle (2.5pt);
\draw [fill](1,-4) circle (2.5pt);
\draw [fill](1.5,-4) circle (2.5pt);
\draw [fill](2,-4) circle (2.5pt);
%----------------------------------
\draw [fill](5,-2) circle (2.5pt);
\draw [fill](5.5,-2) circle (2.5pt);
\draw [fill](6,-2) circle (2.5pt);
\draw [fill](6.5,-2) circle (2.5pt);
\draw [fill](7,-2) circle (2.5pt);
%----------------------------------
\draw [fill](5,0) circle (2.5pt);
\draw [fill](5.5,0) circle (2.5pt);
\draw [fill](6,0) circle (2.5pt);
\draw [fill](6.5,0) circle (2.5pt);
\draw [fill](7,0) circle (2.5pt);
%----------------------------------
\draw [fill](5,-6) circle (2.5pt);
\draw [fill](5.5,-6) circle (2.5pt);
\draw [fill](6,-6) circle (2.5pt);
\draw [fill](6.5,-6) circle (2.5pt);
\draw [fill](7,-6) circle (2.5pt);
%-------------------------------------
\draw [fill](0,-6) circle (2.5pt);
\draw [fill](0.5,-6) circle (2.5pt);
\draw [fill](1,-6) circle (2.5pt);
\draw [fill](1.5,-6) circle (2.5pt);
\draw [fill](2,-6) circle (2.5pt);
\node[above] at (0,0){$1$};
\node[above] at (0.5,0){$2$};
\node[above] at (1,0){$3$};
\node[above] at (1.5,0){$4$};
\node[above] at (2,0){$5$};
%-----------------------------
\node[above] at (5,0){$1$};
\node[above] at (5.5,0){$2$};
\node[above] at (6,0){$3$};
\node[above] at (6.5,0){$4$};
\node[above] at (7,0){$5$};
%----------------------------
\node[below] at (0,-6){$1$};
\node[below] at (0.5,-6){$2$};
\node[below] at (1,-6){$3$};
\node[below] at (1.5,-6){$4$};
\node[below] at (2,-6){$5$};
%---------------------------
\node[below] at (5,-6){$1$};
\node[below] at (5.5,-6){$2$};
\node[below] at (6,-6){$3$};
\node[below] at (6.5,-6){$4$};
\node[below] at (7,-6){$5$};
\end{tikzpicture}
\end{center}
\caption{Here the rewriting procedure $\Omega R_a = \widetilde{R_a}\Omega$ is shown. The braid $R_a$ takes all missing crossings and $R_a$ is transformed to $\Omega$. Meanwhile, $\Omega$ is transformed  to the $\widetilde{R_a}$.}\label{RO->}
\end{figure}
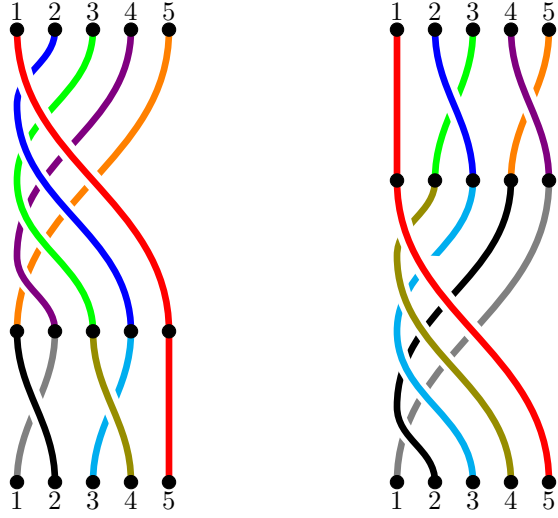

\smallskip

\begin{lemma}\label{fliplemma}
For any two permutations $a$ and $b$, we have
\[
\widetilde{a \asymp b} = \widetilde{a} \asymp \widetilde{b}, \qquad \widetilde{a \bowtie b} = \widetilde{a} \bowtie \widetilde{b}.
\]
\end{lemma}
\begin{proof}
This immediately follows from Theorem \ref{GSB}. Indeed, we have
\[
 \xymatrix{
 & [\Omega|R_a|R_b] \ar@{->}[rd] \ar@{->}[ld] & \\
 [\widetilde{R_a}|\Omega|R_b] \ar@{->}[d] && [\Omega| R_a \asymp R_b|R_a \bowtie R_b] \ar@{->}[d]\\
 [\widetilde{R_a}|\widetilde{R_b}|\Omega] \ar@{->}[d] && [\widetilde{R_a \asymp R_b}|\Omega|R_a \bowtie R_b] \ar@{->}[d]\\
 [\widetilde{R_a} \asymp \widetilde{R_b}|\widetilde{R_a} \bowtie \widetilde{R_b}|\Omega] \ar@{=}[rr] && [\widetilde{R_a \asymp R_b}|\widetilde{R_a \bowtie R_b}|\Omega]
  }
\]
it follows that $\widetilde{a \asymp b} = \widetilde{a} \asymp \widetilde{b}$,and $\widetilde{a \bowtie b} = \widetilde{a} \bowtie \widetilde{b}$, as claimed.
\end{proof}

\smallskip

\par Now we can present a Gr\"obner --- Shirshov basis for the braid groups.

\begin{theorem}
For the braid group $\mathbf{Br}_n$, which is generated by non-repeating braids (Thurston's generators) $R_a$, $a \in \mathbb{S}_n$, a Gr\"obner --- Shirshov basis, with respect the order $\preceq$ (see Definition \ref{order}), consists of the following relations:
\begin{align}
& R_aR_b = R_{a \asymp b}R_{a \bowtie b}\\
&\Omega^{-1}R_a  = \widetilde{R_a}\Omega^{-1},\\
&\Omega\Omega^{-1} = \Omega^{-1}\Omega = R_\varepsilon.
\end{align}
\end{theorem}
\begin{proof}
We just have to check the last two relations, but the second relation follows from Lemma \ref{fliplemma}. Since $\widetilde{\widetilde R} = R$, then for a fixed $\delta \in \{-1,1\}$, we get
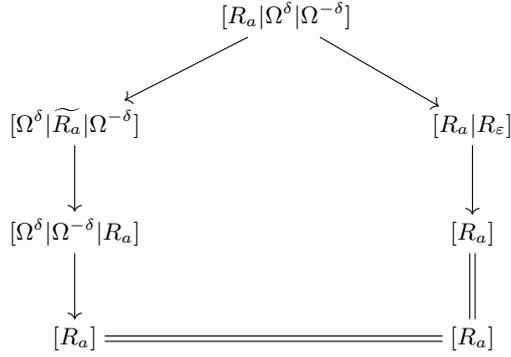
\[
 \xymatrix{
 & [R_a|\Omega^\delta|\Omega^{-\delta}] \ar@{->}[rd] \ar@{->}[ld] &\\
 [\Omega^\delta|\widetilde{R_a}|\Omega^{-\delta}] \ar@{->}[d] && [R_a|R_\varepsilon] \ar@{->}[d]\\
 [\Omega^{\delta}|\Omega^{-\delta}|R_a] \ar@{->}[d] && [R_a] \ar@{=}[d] \\
 [R_a] \ar@{=}[rr] && [R_a]
 }
\]
it means that all relations are closed, as claimed.
\end{proof}

\smallskip

\begin{remark}
A Gr\"obner --- Shirshov basis for the braid groups has been already found before  (see \cite{BokutBraid}). This result is based on the concept of Bokut' --- Shiao's normal form for permutations \cite{BSh}. It was showed that a Gr\"obner --- Shirshov basis for the braid monoid and for the braid groups are described via two and five kinds of relations, respectively. But, of course, all these relations can be described via Thurston's operations and, in fact, most of them have the same spirit, i.e., we have the same relations as we described.
\end{remark}

\section{Thurston's algorithm for rewriting braids to the greedy normal form.}
In this section we will describe (step by step) the working of Thurston's automaton via an algorithm and also describe the output words. We will also present an example and see that this algorithm can be easily used for rewriting braids to the normal form (greedy normal form).

\smallskip

\begin{theorem}\label{cononform}
Let $R_w = R_{b_1} \cdots R_{b_\ell} \in \mathbf{Br}_n^+$ be a positive braid word in the greedy normal form, then the final state of $M(R_a, R_w)$ can be described as follows:
\begin{multline*}
R_a\cdot R_w \to \\ \to  [R_a \asymp R_{b_1}][(R_a \bowtie R_{b_1})\asymp R_{b_2}]\cdots [(((R_a \bowtie R_{b_1})\bowtie R_{b_2}) \bowtie \cdots \bowtie R_{b_{\ell-1}})\asymp R_{b_\ell}][(R_a \bowtie R_{b_1})\bowtie \cdots \bowtie R_{b_\ell}].
\end{multline*}
\end{theorem}
\begin{proof} First of all, from Proposition \ref{stop} it follows that we can rewrite our word strictly in one direction (from left to right), i.e, we cannot come back after some steps in the chosen direction and etc. We have
\begin{multline*}
R_a \cdot R_w = R_a \cdot  R_{b_1} \cdots R_{b_\ell} \to [R_a \asymp R_{b_1}][R_a \bowtie R_{b_1}] \cdot R_{b_2} \cdots R_{b_\ell} \to \\ \to [R_a \asymp R_{b_1}][(R_a \bowtie R_{b_1})\asymp R_{b_2}][(R_a \bowtie R_{b_1})\bowtie R_{b_2}] \cdot R_{b_3} \cdots R_{b_\ell}.
\end{multline*}
Let us remark that $[R_a \asymp R_{b_1}][(R_a \bowtie R_{b_1})\asymp R_{b_2}]$ is the greedy normal form. Indeed, since $R_{b_1}R_{b_2}$ is the greedy normal form, $R_{b_1} \asymp R_{b_2} = R_{b_1}$, $R_{b_1} \bowtie R_{b_2} = R_{b_2}$. Then, using (\ref{l}), we get
\[
[R_a \asymp R_{b_1}]\asymp[(R_a \bowtie R_{b_1})\asymp R_{b_2}] = R_a \asymp (R_{b_1} \asymp R_{b_2}) = R_a \asymp R_{b_1},
\]
as claimed. Further, using induction, we will obtain
\begin{multline*}
R_a\cdot R_w \to \\ \to [R_a \asymp R_{b_1}][(R_a \bowtie R_{b_1})\asymp R_{b_2}]\cdots [(((R_a \bowtie R_{b_1})\bowtie R_{b_2}) \bowtie \cdots \bowtie R_{b_{j-1}})\bowtie R_{b_j}] R_{b_{j+1}}\cdots R_{b_\ell}.
\end{multline*}
Let us again remark that $[((R_a\bowtie R_{b_1})\bowtie \cdots \bowtie R_{b_{j-2}})\asymp R_{b_{j-1}}][((R_a \bowtie R_{b_1})\bowtie R_{b_{j-1}})\asymp R_{b_j}]$ is the greedy normal form. Indeed, let us denote $(R_a\bowtie R_{b_1})\bowtie R_{b_{j-2}}$ by $R_c$, then, using (\ref{l}), we arrive at
\[
[R_c \asymp R_{b_{j-1}}] \asymp [(R_c\bowtie R_{b_{j-1}})\asymp R_{b_j}] = R_c \asymp (R_{b_{j-1}} \asymp R_{b_{j}}) = R_c \asymp R_{b_{j-1}},
\]
as claimed. So, using induction, we will get
\begin{multline*}
R_a\cdot R_w \to \\ \to  [R_a \asymp R_{b_1}][(R_a \bowtie R_{b_1})\asymp R_{b_2}]\cdots [(((R_a \bowtie R_{b_1})\bowtie R_{b_2}) \bowtie \cdots \bowtie R_{b_{\ell-1}})\asymp R_{b_\ell}][(R_a \bowtie R_{b_1})\bowtie \cdots \bowtie R_{b_\ell}],
\end{multline*}
as claimed.
\end{proof}

\smallskip

\par Now, we can present the working of Thurston's automaton via the following algorithm.

\smallskip

\begin{itemize}
 \item[{\bf Step 0.}] There are two permutations $a$ and $b$.
 \item[{\bf Step 1.}] Find $a^{-1}$, find $\omega b$ and put $a':=a^{-1}$ and $b': =b \omega$.
 \item[{\bf Step 2.}] Construct the sets $R_{a'}$ and $R_{b'}$ (see Definition \ref{conR}).
 \item[{\bf Step 3.}] Find  the set $R_x: = R_{a'} \cap R_{b'}$. Check whether or not the set $R_x$ satisfies the condition ii) of Lemma \ref{criteria}. If yes, then go to {\bf Step 4}, otherwise $R_aR_b$ is the greedy normal form and go to {\bf Step 6}.
 \item[{\bf Step 4.}] Find the permutation $x$ using Definition \ref{conR}.
 \item[{\bf Step 5.}] Put $a \asymp b: = ax$ and $a \bowtie b:= x^{-1}b$, i.e., $R_a R_b \to R_{a \asymp b}R_{a \bowtie b}$.
 \item[{\bf Step 6.}] The greedy normal form is found. The end.
\end{itemize}

\smallskip

\par Let us demonstrate the working of this algorithm.

\smallskip

\begin{example}
Let us consider the following permutations,
\[
a= \begin{pmatrix} 1&2&3&4&5&6&7&8\\3&1&7&8&4&5&2&6 \end{pmatrix}, \qquad b= \begin{pmatrix} 1&2&3&4&5&6&7&8\\5&2&6&7&8&1&4&3 \end{pmatrix}.
\]
the corresponding braids $R_a$ and $R_b$ are shown in the fig. \ref{ex1alg}.

\begin{figure}[h!]
\begin{center}
\begin{tikzpicture}
%---------------------------------------------------------------------------------------------------
%left braid
%---------------------------------------------------------------------------------------------------
%strands
%second
\draw [line width = 2.5] (0.5,0) to [out = 270, in = 70] (0,-2) to [out = 250,in=90] (0,-5);
%seventh
\draw [line width = 2.5] (3,0) to [out = 280, in = 80] (0.3,-3.5) to [out=260,in=90]   (0.5,-5);
%eight
\draw [line width = 2.5] (3.5,0) to [out = 270, in = 70] (3,-3.3) to [out = 260,in=110](2.5,-5);
%sixth
\draw [line width = 9,white] (2.5,0) to [out = 270, in = 90] (3,-1.5) to [out = 270, in=110] (2,-5);
\draw [line width = 2.5] (2.5,0) to [out = 270, in = 90] (3,-1.5) to [out = 270, in=110] (2,-5);
%fifth
\draw [line width = 8,white] (2,0) to [out = 250, in = 110] (0.9,-2) to [out =290, in =90] (1.5,-5);
\draw [line width = 2.5] (2,0) to [out = 250, in = 110] (0.9,-2) to [out =290, in =90] (1.5,-5);
%fourth
\draw [line width = 9,white] (1.5,0) to [out = 280, in = 85] (3.5,-5);
\draw [line width = 2.5] (1.5,0) to [out = 280, in = 85] (3.5,-5);
%third
\draw [line width = 9, white] (1,0) to [out = 280, in = 120] (2.25,-3.7) to [out = 300,in = 90](3,-5);
\draw [line width = 2.5] (1,0) to [out = 280, in = 120] (2.25,-3.7) to [out = 300,in = 90](3,-5);
%first
\draw [line width = 9,white] (0,0) to [out = 270, in = 90] (1,-5);
\draw [line width = 2.5] (0,0) to [out = 270, in = 90] (1,-5);
%top points%
\draw [fill](0,0) circle (2.5pt);
\draw [fill](0.5,0) circle (2.5pt);
\draw [fill](1,0) circle (2.5pt);
\draw [fill](1.5,0) circle (2.5pt);
\draw [fill](2,0) circle (2.5pt);
\draw [fill](2.5,0) circle (2.5pt);
\draw [fill](3,0) circle (2.5pt);
\draw [fill](3.5,0) circle (2.5pt);
%button points%
\draw [fill](0,-5) circle (2.5pt);
\draw [fill](0.5,-5) circle (2.5pt);
\draw [fill](1,-5) circle (2.5pt);
\draw [fill](1.5,-5) circle (2.5pt);
\draw [fill](2,-5) circle (2.5pt);
\draw [fill](2.5,-5) circle (2.5pt);
\draw [fill](3,-5) circle (2.5pt);
\draw [fill](3.5,-5) circle (2.5pt);
%numbers
%top
\node [above] at (0,0.1) {$1$};
\node [above] at (0.5,0.1) {$2$};
\node [above] at (1,0.1) {$3$};
\node [above] at (1.5,0.1) {$4$};
\node [above] at (2,0.1) {$5$};
\node [above] at (2.5,0.1) {$6$};
\node [above] at (3,0.1) {$7$};
\node [above] at (3.5,0.1) {$8$};
%button
\node[below] at (0,-5.1){$1$};
\node[below] at (0.5,-5.1){$2$};
\node[below] at (1,-5.1){$3$};
\node[below] at (1.5,-5.1){$4$};
\node[below] at (2,-5.1){$5$};
\node[below] at (2.5,-5.1){$6$};
\node[below] at (3,-5.1){$7$};
\node[below] at (3.5,-5.1){$8$};
%--------------------------------------------
%right braid
%--------------------------------------------
%strands
%sixth
\draw [line width = 2.5] (9.5,0) to [out = 270, in = 80] (6.7,-3) to [out = 260, in = 80] (7,-5);
%eight
\draw [line width = 2.5] (10.5,0) to [out = 270, in = 90]  (8,-5);
%seventh
\draw [line width = 9,white] (10,0) to [out = 270, in = 80] (7.7,-3.5) to [out = 260,in = 90] (8.5,-5);
\draw [line width = 2.5] (10,0) to [out = 270, in = 80] (7.7,-3.5) to [out = 260,in = 90] (8.5,-5);
%fifth
\draw [line width = 9,white] (9,0) to [out = 270, in = 90] (10.5,-2.5) to [out = 270,in=110]  (10.5,-5);
\draw [line width = 2.5] (9,0) to [out = 270, in = 90] (10.5,-2.5) to [out = 270,in=110]  (10.5,-5);
%fourth
\draw [line width = 9,white] (8.5,0) to [out = 270, in = 90] (9.7,-3.5) to [out = 270,in=105]  (10,-5);
\draw [line width = 2.5] (8.5,0) to [out = 270, in = 90] (9.7,-3.5) to [out = 270,in=105]  (10,-5);
%third
\draw [line width = 9,white] (8,0) to [out = 260,in=90]  (9.5,-5);
\draw [line width = 2.5] (8,0) to [out = 260,in=90]  (9.5,-5);
%second
\draw [line width = 9,white] (7.5,0) to [out = 280, in = 90] (6.7,-1.8) to [out=270,in=80] (7.5,-5);
\draw [line width = 2.5] (7.5,0) to [out = 280, in = 90] (6.7,-1.8) to [out=270,in=80] (7.5,-5);
%first
\draw [line width = 9,white] (7,0) to [out = 280, in = 130] (8.1,-3.3) to [out = 310, in = 90] (9,-5);
\draw [line width = 2.5] (7,0) to [out = 280, in = 130] (8.1,-3.3) to [out = 310, in = 90] (9,-5);
%top points%
\draw [fill](7,0) circle (2.5pt);
\draw [fill](7.5,0) circle (2.5pt);
\draw [fill](8,0) circle (2.5pt);
\draw [fill](8.5,0) circle (2.5pt);
\draw [fill](9,0) circle (2.5pt);
\draw [fill](9.5,0) circle (2.5pt);
\draw [fill](10,0) circle (2.5pt);
\draw [fill](10.5,0) circle (2.5pt);
%button points%
\draw [fill](7,-5) circle (2.5pt);
\draw [fill](7.5,-5) circle (2.5pt);
\draw [fill](8,-5) circle (2.5pt);
\draw [fill](8.5,-5) circle (2.5pt);
\draw [fill](9,-5) circle (2.5pt);
\draw [fill](9.5,-5) circle (2.5pt);
\draw [fill](10,-5) circle (2.5pt);
\draw [fill](10.5,-5) circle (2.5pt);
%numbers
%top
\node [above] at (7,0.1) {$1$};
\node [above] at (7.5,0.1) {$2$};
\node [above] at (8,0.1) {$3$};
\node [above] at (8.5,0.1) {$4$};
\node [above] at (9,0.1) {$5$};
\node [above] at (9.5,0.1) {$6$};
\node [above] at (10,0.1) {$7$};
\node [above] at (10.5,0.1) {$8$};
%button
\node[below] at (7,-5.1){$1$};
\node[below] at (7.5,-5.1){$2$};
\node[below] at (8,-5.1){$3$};
\node[below] at (8.5,-5.1){$4$};
\node[below] at (9,-5.1){$5$};
\node[below] at (9.5,-5.1){$6$};
\node[below] at (10,-5.1){$7$};
\node[below] at (10.5,-5.1){$8$};
\end{tikzpicture}
\caption{Here are shown the braids $R_a$ (left side) and $R_b$ (right side) which correspond to the permutations $a$ and $b$, respectively.}\label{ex1alg}
\end{center}
\end{figure}
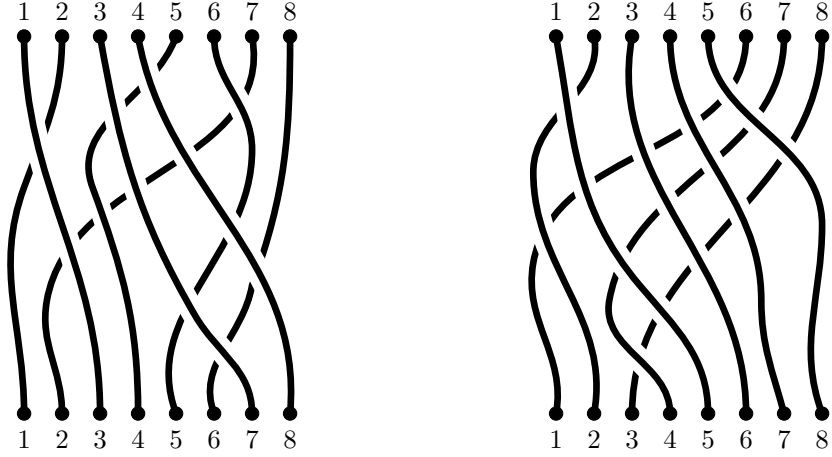
\smallskip

\begin{itemize}
 \item[{\bf Step 0.}] {\bf START}
 \item[{\bf Step 1.}] We get $a' = a^{-1} = \begin{pmatrix} 1&2&3&4&5&6&7&8\\2&7&1&5&6&8&3&4 \end{pmatrix}$, and $b' = b \omega$,
     $$b' = \begin{pmatrix} 1&2&3&4&5&6&7&8\\5&2&6&7&8&1&4&3 \end{pmatrix}\begin{pmatrix} 1&2&3&4&5&6&7&8\\8&7&6&5&4&3&2&1 \end{pmatrix} = \begin{pmatrix} 1&2&3&4&5&6&7&8\\4&7&3&2&1&8&5&6 \end{pmatrix},
     $$
 \item[{\bf Step 2.}] Using Definition \ref{conR}, we get
  \begin{align*}
   &aR_a = \{(1,3), (2,3), (2,4),(2,5),(2,7),(2,8),(4,7),(4,8), (5,7),(5,8),(6,7),(6,8)\},\\
   &\neg R_b = \{(1,3),(1,4),(1,5),(2,3),(2,4),(2,5),(2,7),(2,8),(3,4),(3,5),(4,5),(6,7),(6,8)\}.
   \end{align*}
  \item[{\bf Step 3.}] Let us find the intersection
   \[
    aR_a \cap \neg R_b = \{(1,3),(2,3),(2,4),(2,5),(2,7),(2,8),(6,7),(6,8)\},
   \]
    we see that this set satisfies to condition ii) of Lemma \ref{criteria}.
  \item[{\bf Step 4.}] Let $R_x = aR_a \wedge \neg R_b = \{(1,3),(2,3),(2,4),(2,5),(2,7),(2,8),(6,7),(6,8)\}$, we have for the permutation $x$ the following system of inequalities
      \[
       \begin{cases}
       x(1) < x(2), \, x(1) > x(3), \, x(1) < x(4), \, x(1) < x(5), \, x(1) < x(6), \, x(1) < x(7), \, x(1) < x(8),\\
       \phantom{x(2) < x(3), \, } x(2) > x(3), \, x(2) > x(4), \, x(2) > x(5), \, x(2) < x(6), \, x(2) > x(7), \, x(2) > x(8),\\
       \phantom{x(2) < x(3), \, x(2) < x(3), \, } x(3) < x(4), \, x(3) < x(5), \, x(3) < x(6), \, x(3) < x(7), \, x(3) < x(8),\\
       \phantom{x(2) < x(3), \, x(2) < x(3), \,  x(2) < x(3), \,}  x(4) < x(5), \, x(4) < x(6), \, x(4) < x(7), \, x(4) < x(8),\\
       \phantom{x(2) < x(3), \, x(2) < x(3), \,  x(2) < x(3), \, x(2) < x(3), \,}  x(5) < x(6), \, x(5) < x(7), \, x(5) < x(8),\\
       \phantom{x(2) < x(3), \, x(2) < x(3), \,  x(2) < x(3), \, x(2) < x(3), \, x(2) < x(3), \,}  x(6) > x(7), \, x(6) > x(8),\\
       \phantom{x(2) < x(3), \, x(2) < x(3), \,  x(2) < x(3), \, x(2) < x(3), \, x(2) < x(3), \, x(2) < x(3), \,}  x(7) < x(8),\\
       \end{cases}
      \]
  it follows that
   \[
   x = \begin{pmatrix} 1&2&3&4&5&6&7&8 \\ 2&7&1&3&4&8&5&6\end{pmatrix}, \qquad x^{-1} = \begin{pmatrix} 1&2&3&4&5&6&7&8 \\ 3&1&4&5&7&8&2&6\end{pmatrix}.
   \]
  \item[{\bf Step 5.}] We get
   \begin{align*}
    &a \asymp b = ax = \begin{pmatrix} 1&2&3&4&5&6&7&8\\3&1&7&8&4&5&2&6 \end{pmatrix} \begin{pmatrix} 1&2&3&4&5&6&7&8 \\ 2&7&1&3&4&8&5&6\end{pmatrix} = \begin{pmatrix} 1&2&3&4&5&6&7&8\\1&2&5&6&3&4&7&8 \end{pmatrix},\\
    & a \bowtie b = x^{-1}b = \begin{pmatrix} 1&2&3&4&5&6&7&8 \\ 3&1&4&5&7&8&2&6\end{pmatrix}\begin{pmatrix} 1&2&3&4&5&6&7&8\\5&2&6&7&8&1&4&3 \end{pmatrix} = \begin{pmatrix} 1&2&3&4&5&6&7&8\\6&5&7&8&4&3&2&1 \end{pmatrix},
   \end{align*}
   in the following picture (see fig. \ref{ww}) we show this procedure via the braid diagrams.
   \item[{\bf Step 6.}] {\bf THE END.}
 \end{itemize}
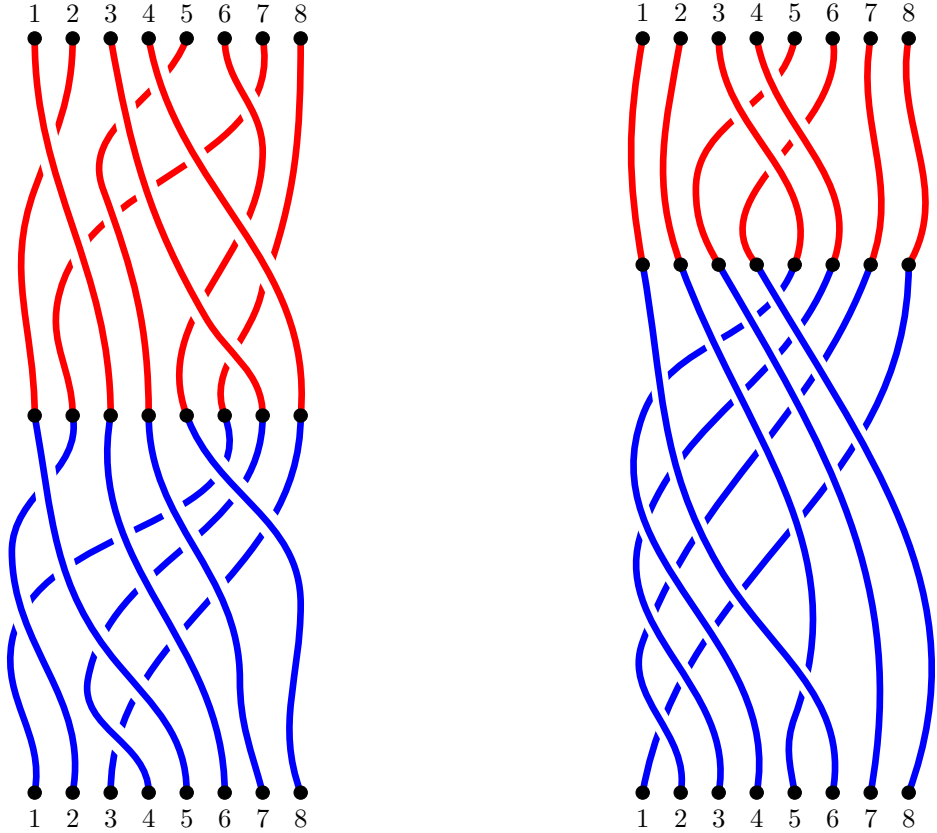
\begin{figure}[h!]
\begin{center}
\begin{tikzpicture}
%ïîñëåäíèé ðèñóíîê â ñòàòüå
%---------------------------------------------------------------------------------------------------
%top braid
%---------------------------------------------------------------------------------------------------
%strands
%second
\draw [line width = 2.5,red] (0.5,0) to [out = 270, in = 70] (0,-2) to [out = 250,in=90] (0,-5);
%seventh
\draw [line width = 2.5,red] (3,0) to [out = 280, in = 80] (0.3,-3.5) to [out=260,in=90]   (0.5,-5);
%eight
\draw [line width = 2.5,red] (3.5,0) to [out = 270, in = 70] (3,-3.3) to [out = 260,in=110](2.5,-5);
%sixth
\draw [line width = 9,white] (2.5,0) to [out = 270, in = 90] (3,-1.5) to [out = 270, in=110] (2,-5);
\draw [line width = 2.5,red] (2.5,0) to [out = 270, in = 90] (3,-1.5) to [out = 270, in=110] (2,-5);
%fifth
\draw [line width = 8,white] (2,0) to [out = 250, in = 110] (0.9,-2) to [out =290, in =90] (1.5,-5);
\draw [line width = 2.5,red] (2,0) to [out = 250, in = 110] (0.9,-2) to [out =290, in =90] (1.5,-5);
%fourth
\draw [line width = 9,white] (1.5,0) to [out = 280, in = 85] (3.5,-5);
\draw [line width = 2.5,red] (1.5,0) to [out = 280, in = 85] (3.5,-5);
%third
\draw [line width = 9, white] (1,0) to [out = 280, in = 120] (2.25,-3.7) to [out = 300,in = 90](3,-5);
\draw [line width = 2.5,red] (1,0) to [out = 280, in = 120] (2.25,-3.7) to [out = 300,in = 90](3,-5);
%first
\draw [line width = 9,white] (0,0) to [out = 270, in = 90] (1,-5);
\draw [line width = 2.5,red] (0,0) to [out = 270, in = 90] (1,-5);
%top points%
\draw [fill](0,0) circle (2.5pt);
\draw [fill](0.5,0) circle (2.5pt);
\draw [fill](1,0) circle (2.5pt);
\draw [fill](1.5,0) circle (2.5pt);
\draw [fill](2,0) circle (2.5pt);
\draw [fill](2.5,0) circle (2.5pt);
\draw [fill](3,0) circle (2.5pt);
\draw [fill](3.5,0) circle (2.5pt);
%numbers
%top
\node [above] at (0,0.1) {$1$};
\node [above] at (0.5,0.1) {$2$};
\node [above] at (1,0.1) {$3$};
\node [above] at (1.5,0.1) {$4$};
\node [above] at (2,0.1) {$5$};
\node [above] at (2.5,0.1) {$6$};
\node [above] at (3,0.1) {$7$};
\node [above] at (3.5,0.1) {$8$};
%--------------------------------------------
%button braid
%--------------------------------------------
%strands
%sixth
\draw [line width = 2.5,blue] (2.5,-5) to [out = 290, in = 80] (-0.3,-8) to [out = 260, in = 80] (0,-10);
%eight
\draw [line width = 2.5,blue] (3.5,-5) to [out = 270, in = 90]  (1,-10);
%seventh
\draw [line width = 9,white] (3,-5) to [out = 270, in = 80] (0.7,-8.5) to [out = 260,in = 90] (1.5,-10);
\draw [line width = 2.5,blue] (3,-5) to [out = 270, in = 80] (0.7,-8.5) to [out = 260,in = 90] (1.5,-10);
%fifth
\draw [line width = 9,white] (2,-5) to [out = 290, in = 90] (3.5,-7.5) to [out = 270,in=110]  (3.5,-10);
\draw [line width = 2.5,blue] (2,-5) to [out = 290, in = 90] (3.5,-7.5) to [out = 270,in=110]  (3.5,-10);
%fourth
\draw [line width = 9,white] (1.5,-5) to [out = 270, in = 90] (2.7,-8.5) to [out = 270,in=105]  (3,-10);
\draw [line width = 2.5,blue] (1.5,-5) to [out = 270, in = 90] (2.7,-8.5) to [out = 270,in=105]  (3,-10);
%third
\draw [line width = 9,white] (1,-5) to [out = 260,in=90]  (2.5,-10);
\draw [line width = 2.5,blue] (1,-5) to [out = 260,in=90]  (2.5,-10);
%second
\draw [line width = 9,white] (0.5,-5) to [out = 280, in = 90] (-0.3,-6.8) to [out=270,in=80] (0.5,-10);
\draw [line width = 2.5,blue] (0.5,-5) to [out = 280, in = 90] (-0.3,-6.8) to [out=270,in=80] (0.5,-10);
%first
\draw [line width = 9,white] (0,-5) to [out = 280, in = 130] (1.1,-8.3) to [out = 310, in = 90] (2,-10);
\draw [line width = 2.5,blue] (0,-5) to [out = 280, in = 130] (1.1,-8.3) to [out = 310, in = 90] (2,-10);
%middle points
%button points%
\draw [fill](0,-5) circle (2.5pt);
\draw [fill](0.5,-5) circle (2.5pt);
\draw [fill](1,-5) circle (2.5pt);
\draw [fill](1.5,-5) circle (2.5pt);
\draw [fill](2,-5) circle (2.5pt);
\draw [fill](2.5,-5) circle (2.5pt);
\draw [fill](3,-5) circle (2.5pt);
\draw [fill](3.5,-5) circle (2.5pt);
%button points%
\draw [fill](0,-10) circle (2.5pt);
\draw [fill](0.5,-10) circle (2.5pt);
\draw [fill](1,-10) circle (2.5pt);
\draw [fill](1.5,-10) circle (2.5pt);
\draw [fill](2,-10) circle (2.5pt);
\draw [fill](2.5,-10) circle (2.5pt);
\draw [fill](3,-10) circle (2.5pt);
\draw [fill](3.5,-10) circle (2.5pt);
%numbers
%button
\node[below] at (0,-10.1){$1$};
\node[below] at (0.5,-10.1){$2$};
\node[below] at (1,-10.1){$3$};
\node[below] at (1.5,-10.1){$4$};
\node[below] at (2,-10.1){$5$};
\node[below] at (2.5,-10.1){$6$};
\node[below] at (3,-10.1){$7$};
\node[below] at (3.5,-10.1){$8$};
%+++++++++++++++++++++++++++++++++++++++++++++++++++++++++++++++++++++++++++++++++
%+++++++++++++++++++++++++++++++++++++++++++++++++++++++++++++++++++++++++++++++++
%top
%strands
%first
\draw [line width = 2.5,red] (8,0) to [out = 260,in=100]  (8,-3);
%second
\draw [line width = 2.5,red] (8.5,0) to [out = 260,in=110]  (8.5,-3);
%fifth
\draw [line width = 2.5,red] (10,0) to [out = 260,in=90]  (8.7,-2) to [out =270, in =120] (9,-3);
%sixth
\draw [line width = 2.5,red] (10.5,0) to [out = 280,in=130] (9.5,-3);
%third
\draw [line width = 9,white] (9,0) to [out = 270,in=70] (10,-3);
\draw [line width = 2.5,red] (9,0) to [out = 270,in=70] (10,-3);
%fourth
\draw [line width = 9,white] (9.5,0) to [out = 280,in=70] (10.5,-3);
\draw [line width = 2.5,red] (9.5,0) to [out = 280,in=70] (10.5,-3);
%seventh
\draw [line width = 2.5,red] (11,0) to [out = 260,in=70] (11,-3);
%eight
\draw [line width = 2.5,red] (11.5,0) to [out = 260,in=60] (11.5,-3);
%-------------------------------------------------------
%button
%---------------------------------------------------------
%eight
\draw [line width = 2.5,blue] (11.5,-3) to [out = 270,in=80] (8,-10);
%seventh
\draw [line width = 7,white] (11,-3) to [out = 250,in=70] (8,-8) to [out =250,in= 80] (8.5,-10);
\draw [line width = 2.5,blue] (11,-3) to [out = 250,in=70] (8,-8) to [out =250,in= 80] (8.5,-10);
%sixth
\draw [line width = 7,white] (10.5,-3) to [out = 250,in=70] (8,-6.5) to [out =250,in= 80] (9,-10);
\draw [line width = 2.5,blue] (10.5,-3) to [out = 250,in=70] (8,-6.5) to [out =250,in= 80] (9,-10);
%fifth
\draw [line width = 7,white] (10,-3) to [out = 250,in=70] (8,-5) to [out =250,in= 80] (9.5,-10);
\draw [line width = 2.5,blue] (10,-3) to [out = 250,in=70] (8,-5) to [out =250,in= 80] (9.5,-10);
%fourth
\draw [line width = 7,white] (9.5,-3) to [out = 300,in=70] (11.5,-10);
\draw [line width = 2.5,blue] (9.5,-3) to [out = 300,in=70] (11.5,-10);
%third
\draw [line width = 7,white] (9,-3) to [out = 300,in=80] (11,-10);
\draw [line width = 2.5,blue] (9,-3) to [out = 300,in=80] (11,-10);
%second
\draw [line width = 7,white] (8.5,-3) to [out = 290,in=70] (10,-9) to [out = 250,in=100] (10,-10);
\draw [line width = 2.5,blue] (8.5,-3) to [out = 290,in=70] (10,-9) to [out = 250,in=100] (10,-10);
%first
\draw [line width = 7,white] (8,-3) to [out = 280,in=120] (9,-7) to[out = 300,in = 80] (10.5,-10);
\draw [line width = 2.5,blue] (8,-3) to [out = 280,in=120] (9,-7) to[out = 300,in = 80] (10.5,-10);
%--------------------------------------------------
%points
%top
\draw [fill](8,0) circle (2.5pt);
\draw [fill](8.5,0) circle (2.5pt);
\draw [fill](9,0) circle (2.5pt);
\draw [fill](9.5,0) circle (2.5pt);
\draw [fill](10,0) circle (2.5pt);
\draw [fill](10.5,0) circle (2.5pt);
\draw [fill](11,0) circle (2.5pt);
\draw [fill](11.5,0) circle (2.5pt);
%middle
\draw [fill](8,-3) circle (2.5pt);
\draw [fill](8.5,-3) circle (2.5pt);
\draw [fill](9,-3) circle (2.5pt);
\draw [fill](9.5,-3) circle (2.5pt);
\draw [fill](10,-3) circle (2.5pt);
\draw [fill](10.5,-3) circle (2.5pt);
\draw [fill](11,-3) circle (2.5pt);
\draw [fill](11.5,-3) circle (2.5pt);
%button
\draw [fill](8,-10) circle (2.5pt);
\draw [fill](8.5,-10) circle (2.5pt);
\draw [fill](9,-10) circle (2.5pt);
\draw [fill](9.5,-10) circle (2.5pt);
\draw [fill](10,-10) circle (2.5pt);
\draw [fill](10.5,-10) circle (2.5pt);
\draw [fill](11,-10) circle (2.5pt);
\draw [fill](11.5,-10) circle (2.5pt);
\node [above] at (8,0.1) {$1$};
\node [above] at (8.5,0.1) {$2$};
\node [above] at (9,0.1) {$3$};
\node [above] at (9.5,0.1) {$4$};
\node [above] at (10,0.1) {$5$};
\node [above] at (10.5,0.1) {$6$};
\node [above] at (11,0.1) {$7$};
\node [above] at (11.5,0.1) {$8$};
\node[below] at (8,-10.1){$1$};
\node[below] at (8.5,-10.1){$2$};
\node[below] at (9,-10.1){$3$};
\node[below] at (9.5,-10.1){$4$};
\node[below] at (10,-10.1){$5$};
\node[below] at (10.5,-10.1){$6$};
\node[below] at (11,-10.1){$7$};
\node[below] at (11.5,-10.1){$8$};
\end{tikzpicture}
\end{center}
\caption{Here are shown the rewriting of the word $R_aR_b$ to the greedy normal form.}\label{ww}
\end{figure}
\end{example}

\section*{Conclusions and Thanks}
We have seen that Thurston's point of view on the braids is very useful and very easy for understanding. The author hopes that the results of this paper will be helpful for studying the conjugacy problem and for calculating (in an explicit form as it was done by V. I. Arnold) the cohomologies of the braid groups. All these questions are interesting, and the author is going to study these problems in the future papers.

\smallskip

\paragraph{Acknowledgements.} The author would like to express his deepest gratitude to Professor Leonid A. Bokut', who has drawn the author's attention to the braid groups. I am also extremely indebted to my friend (my Chinese Brother) Zhang Junhuai for the great support, without which the author's life would be very difficult.

\end{document}